\documentclass{SCIYA2017enOL}
\online

\usepackage[all]{xy}    

\usepackage{float}
\usepackage{tikz}
\usetikzlibrary{matrix,arrows}
\usepackage{graphicx}

\newcounter{RomanNumber}
\newcommand{\MyRoman}[1]{\setcounter{RomanNumber}{#1}\Roman{RomanNumber}}

\begingroup
\catcode`\&=13
\gdef\pampmatrix{%
  \begingroup
  \let&=\amsamp
  \begin{pmatrix}%
}
\gdef\endpampmatrix{\end{pmatrix}\endgroup}
\endgroup

\begin{document}

\ensubject{fdsfd}

\ArticleType{ARTICLES}
\Year{2017}
\Month{January}%
\Vol{60}
\No{1}
\BeginPage{1} %
\DOI{XXXXXXXX}
\ReceiveDate{March 11, 2019}
\AcceptDate{May 5, 2019}

\title[Gauge groups of $5$-manifolds]{Homotopy of gauge groups over non-simply-connected five dimensional manifolds}
{}

\author[1,$\ast$]{Ruizhi Huang}{huangrz@amss.ac.cn}

\AuthorMark{Ruizhi Huang}

\AuthorCitation{Ruizhi Huang}

\address[1]{Institute of Mathematics, Academy of Mathematics and Systems Science,\\
Chinese Academy of Sciences, Beijing {\rm100190}, China}

\abstract{Both the gauge groups and $5$-manifolds are important in physics and mathematics. In this paper, we combine them together to study the homotopy aspects of gauge groups over $5$-manifolds.
For principal bundles over non-simply connected oriented closed $5$-manifolds of certain type, we prove various homotopy decompositions of their gauge groups according to different geometric structures on the manifolds, and give partial solution to the classification of the gauge groups.
As applications, we estimate the homotopy exponents of their gauge groups, and show periodicity results of the homotopy groups of gauge groups analogous to Bott periodicity. Our treatments here are also very effective for rational gauge groups in general context, and applicable for higher dimensional manifolds.}
\keywords{gauge group, $5$-manifolds, homotopy decompositions, homotopy exponents, homotopy groups, rational homotopy theory, Bott periodicity}

\MSC{primary 57S05, 55P15, 55P40, 54C35; secondary 55P62, 57R19, 58B05, 81T13}

\maketitle

\tableofcontents

\section{Introduction}
\noindent The topology of $5$-manifolds is wild. An explicit and complete classification of diffeomorphism classes of simply connected $5$-manifolds is developed by Barden \cite{Barden65}. It turns out that the homotopy, homeomorphism and diffeomorphism classification all coincide in this case. 
In contrast, the classification of non-simply connected $5$-manifolds is impossible since it is harder than solving the word problem for groups, while the latter problem is unsolvable. The interesting problems are then to study the classifications of $5$-manifolds with prescribed fundamental groups. However, it turns out that these classifications are in general extremely hard, and there are only a few concrete results in the literature (e.g., Hambleton-Su \cite{HS11}, Kreck-Su \cite{KSu17}).

There is particularly an interesting family of non-simply connected $5$-manifolds constructed naturally from $4$-manifolds.
Explicitly, let $N$ be a simply connected oriented closed $4$-manifold of rank $m$, i.e., the rank of $H^2(N; \mathbb{Z})$ as torsion free abelian group is $m$.
Consider the $S^1$-principal bundle 
\begin{equation}\label{wfibrationdef}
S^1\hookrightarrow W\stackrel{\pi}{\longrightarrow} N
\end{equation}
determined by a cohomology class
\begin{equation}\label{isoelementx}
x\in H^{2}(N;\mathbb{Z})\cong [N,BS^1\simeq K(\mathbb{Z}, 2)\simeq \mathbb{C}P^{\infty}]\cong \oplus_{i=1}^{m}\mathbb{Z}.
\end{equation}
Since $W$ is the sphere bundle of the associated $\mathbb{R}^2$-vector bundle, it is a $5$-dimensional closed manifold. Denote $x=(a_1,a_2,\ldots,a_m)$ by the isomorphisms (\ref{isoelementx}). We may call an element $\alpha \in \oplus_{i=1}^{m}\mathbb{Z}$ \textit{primitive} if the quotient group $\oplus_{i=1}^{m}\mathbb{Z}/\langle \alpha \rangle$ is isomorphic to $\oplus_{i=1}^{m-1}\mathbb{Z}$ \cite{DL05}. Then $x=c \alpha$ for some primitive element $\alpha$, and the fundamental group of $W$ is $\mathbb{Z}/c$ by expecting the long exact sequence of the homotopy groups of (\ref{wfibrationdef}). In particular, if $x$ is primitive then $W$ is simply connected. 
The (topological) classification of such $W$ has only been done for the cases when $W$ is simply connected by  Duan and Liang \cite{DL05}, and when $\pi_1{W}\cong \mathbb{Z}/2$ by Hambleton and Su \cite{HS11}. Surprisingly, even the homotopy classification of $W$ in the other cases is not yet understood.

In this paper, we consider the oriented closed $5$-manifolds $M$ (or more generally Poincar\'{e} duality complexes) such that the cohomology group $H^\ast(M; \mathbb{Z})\cong H^\ast(W;\mathbb{Z})$. Instead of studying the homotopy types of them directly which is supposed to very complicated, we may investigate the homotopy of their gauge groups for various Lie groups. Therefore, we may focus on the homotopy structures of their gauge groups.

The role of gauge theory in topology was nicely summarised in a survey paper of Cohen and Milgram \cite{CM94}. They remarked there that the main question in gauge theory, from an algebraic topological point of view, is to understand the homotopy type of the gauge groups and the moduli spaces of particular connections on the principal bundles. Along this point of view, the homotopy aspects of gauge groups were largely investigated in the last twenty years. In 2000, Crabb and Sutherland \cite{CS00} showed that, though there may be infinitely many inequivalent principal bundles over a finite complex, their gauge groups have only finitely many distinct homotopy types. By using homotopy theoretic techniques, especially the estimation of orders of particular involved Samelson products, the classifications of homotopy types of gauge groups are studied over $S^4$ by many experts including 
Hamanaka, Hasui, Kishimoto, Kono, Sato, So, Theriault, Tsutaya (\cite{Kono91}, \cite{HK06}, \cite{Theriault15}, \cite{HKKS16}, \cite{Theriault17}, \cite{KTT17}, \cite{KK18}, \cite{HKST18}), etc, and also over general $4$-manifolds by Theriault \cite{Theriault10} and So \cite{So16}.

Here we turn to study the homotopy of gauge groups in the first high dimensional case as mentioned. Let us start with some standard terminologies and notations. Throughout the paper, let $G$ be a simply connected compact simple Lie group with $\pi_4(G)=0$, and $P$ be a principal $G$-bundle over $M$. Recall that $H^\ast(M; \mathbb{Z})\cong H^\ast(W;\mathbb{Z})$ and $W$ is the total space of a $S^1$-bundle over a simply connected $4$-manifold of rank $m$.
The \textit{gauge group} $\mathcal{G}(P)$ of $P$, by definition, is the group of $G$-equivariant automorphisms of $P$ that fix $M$. 
The restriction on the triviality of $\pi_4(G)$ is almost empty since it is satisfied by the types of Lie groups in the following list:
\[SU(n+1) (n\geq 2), \ \ Spin(n) (n\geq 6), \ \ G_2, \ \ F_4,\ \ E_6,\ \ E_7,\ \ E_8\]
(and $Sp(n)$ if localized away from $2$ since $\pi_{4}(Sp(n))\cong \mathbb{Z}/2$).
It can be showed that the group of the isomorphism classes of principal $G$-bundles over $M$ is isomorphic to (Lemma \ref{MBG})
\begin{equation}\label{introMBG}
[M, BG]\cong \pi_{1}(M)\cong \mathbb{Z}/c.
\end{equation}
We then may also denote that the gauge group $\mathcal{G}(P)$ by $\mathcal{G}_k(M)$ for $[P]=k\in [M, BG]$ under the isomorphism (\ref{introMBG}). By Gottlieb \cite{Gottlieb72} or Atiyah and Bott \cite{AB83}, there is a homotopy equivalence
\begin{equation}\label{introABG}
B\mathcal{G}_k(M)\simeq {\rm Map}_k(M, BG)
\end{equation}
between the classifying space of $\mathcal{G}_k(M)$ and the component of free mapping space ${\rm Map} (M, BG)$ corresponding to $k$ through (\ref{introMBG}). We then will not distinguish these two spaces since we only care about homotopy types. The first observation is that the homotopy theory of gauge groups over simply connected $5$-manifolds is not interesting. Indeed for any trivial bundle $P=M\times G$ we always have 
\[\mathcal{G}(P)\cong {\rm Map}(M, G)\simeq {\rm Map}^\ast(M, G)\times G.\] 
Hence, from now on we suppose $c\geq 2$.

The first two main results are about the homotopy decompositions of $\mathcal{G}_k(M)$ up to loops. This way of looking at the homotopy aspects of particular spaces is quite useful in algebraic topology, which allows us to understand a complicated object through relatively easier ones.
Denote $P^{m+1}(c)=S^m\cup_c e^{m+1}$ to be the Moore space determined by a degree $c$ map of spheres. Our first theorem is proved based on various computations and techniques in unstable homotopy theory, and a key application of an example of $5$-Poincar\'{e} complex but non-manifold due to Gilter and Stasheff (Madsen and Milgram \cite{MM79}). 
\begin{theorem}[Theorem \ref{gaugedecom6c}, Lemma \ref{mappingmooredes}]\label{theorem1}
Let $M$ be a five dimensional oriented closed manifold with $\pi_{1}(M)\cong \mathbb{Z}/c$ ($6\nmid c$) and $H_{2}(M; \mathbb{Z})$ is torsion free of rank $m-1$. Let $G$ be a simply connected compact simple Lie group with $\pi_4(G)=0$. We have the following homotopy equivalences:
\begin{itemize}
\item if $M$ is a spin manifold,
\[\Omega^2 \mathcal{G}_k(M)\simeq \Omega^2\mathcal{G}_k(P^4(c))\times \Omega^3G\{c\}\times \Omega^7G\times \prod_{i=1}^{m-1} (\Omega^4 G  \times \Omega^5 G );\]
\item if $M$ is a non-spin manifold,
\[\Omega^2 \mathcal{G}_k(M)\simeq \Omega^2\mathcal{G}_k(P^4(c))\times \Omega^3 {\rm Map}_0^\ast(\mathbb{C}P^2, G)\times \Omega^3G\{c\}\times \prod_{i=1}^{m-1} \Omega^4 G  \times \prod_{i=1}^{m-2} \Omega^5 G,\]
\end{itemize}
for each $k\in \mathbb{Z}/c$, where $G\{c\}$ is the homotopy fibre of the power map $c: G\rightarrow G$.
\end{theorem}
Notice that we only consider the decompositions of looped gauge groups rather than themselves. The reason is that after looping the homotopy structures of the gauge groups become much clearer and more accessible, as they can be decomposed into small common pieces. This technique is frequently used in homotopy theory. One classical example is a result of Selick \cite{Selick78} on the homotopy exponent of $3$-sphere at odd prime $p$ where it is necessary to analyze structure gotten after looping $2$ or $3$ times. We should also notice that the looped decompositions in Theorem \ref{theorem1} are the best possible ones, since they are already \textit{completely decomposed} when either $M$ is spin or non-spin, that is, in either case the gauge group is decomposed into a product of iterated loop spaces of $G$ and some extra smallest pieces preserving the nontrivial cohomology operations of $M$. 

There is also another situation when we can get complete decompositions of the $3$-fold looped gauge groups. At this time we may weaken the conditions on the fundamental groups but put extra geometric restrictions on $M$. We remark that all the decompositions in Theorem \ref{theorem1} and \ref{theorem2} are very useful for studying homotopy exponents as we will see in the sequel. Further, our method here is quite general and applicable for other higher dimensional manifolds (for instance, see \cite{Huang18}).
\begin{theorem}[Theorem \ref{gaugedecomPoincare}, Definition \ref{stablepara}, Corollary \ref{gaugedecompimfld}, Lemma \ref{mappingmooredes}]\label{theorem2}
Let $M$ be a five dimensional stably parallelizable oriented closed manifold with precisely one $5$-dimensional cell, $\pi_{1}(M)\cong \mathbb{Z}/c$ ($2\nmid c$) and $H_{2}(M; \mathbb{Z})$ is torsion free of rank $m-1$. Let $G$ be a simply connected compact simple Lie group with $\pi_4(G)=0$. Then we have 
\[\Omega^3 \mathcal{G}_k(M)\simeq \Omega^3\mathcal{G}_k(P^4(c))\times \Omega^4G\{c\}\times \Omega^8G\times \prod_{i=1}^{m-1} (\Omega^5 G  \times \Omega^6 G),\]
for each $k\in \mathbb{Z}/c$.
\end{theorem}

Notice that the gauge groups of $P^4(c)$ appear in all the homotopy decompositions in Theorem \ref{theorem1} and \ref{theorem2} as the only non-trivial factors. Nevertheless, $\mathcal{G}_k(P^4(c))$ is accessible based on the researches on $\mathcal{G}_k(S^4)$ mentioned in the beginning. By (\ref{introABG}), there is a natural fibre sequence
\[
{\rm Map}^{\ast}_0(P^4(c), G)\rightarrow  \mathcal{G}_k(P^4(c))\rightarrow G\stackrel{\partial_k}{\rightarrow} {\rm Map}^{\ast}_k(P^4(c), BG)\rightarrow {\rm Map}_k(P^4(c), BG)\stackrel{{\rm ev}}{\rightarrow} BG,
\] 
such that $\partial_k\simeq k\partial_1$ (Lang \cite{Lang73}; Lemma \ref{langlemma}). The order ${\rm ord}(\partial_1)$ of the connecting map $\partial_1$ plays a key role in the classification of $\mathcal{G}_k(P^4(c))$ in general, which indeed divides ${\rm ord}(\tilde{\partial}_1)$ the order of the connecting map $\tilde{\partial}_1$ for $\mathcal{G}_1(S^4)$ (Lemma \ref{orderconnectionlemma}). The latter order is well understood in many cases as summarised in Table \ref{tables4guageorder} of Section \ref{gaugemooresec}. Let $\nu_p(m)$ be the number of powers of $p$ which divide the integer $m$.
Then with the help of the well known Dirichlet's theorem on arithmetic progressions, we obtain partial classification of looped gauge groups of $M$.

\begin{theorem}[Theorem \ref{gaugedecom6cclass} and \ref{gaugedecomPoincareclass}]\label{theorem2.5}
Let $G$ be a simply connected compact simple Lie group with $\pi_4(G)=0$. Let $M$ be a five dimensional oriented closed manifold with $\pi_{1}(M)\cong \mathbb{Z}/c$ and $H_{2}(M; \mathbb{Z})$ is torsion free of rank $m-1$. Denote $i=2$ if $6\nmid c$, or $i=3$ if $2\nmid c$ and $M$ is stably parallelizable. 
Then for any $k\in \mathbb{Z}/c$, if the greatest common divisor $(k, {\rm ord}(\partial_1), c)=(l, {\rm ord}(\partial_1), c)$ we have for any prime $p$
\[\Omega^i\mathcal{G}_k(M)\simeq_{(p)} \Omega^i\mathcal{G}_l(M).\]
In particular, there are at most $\nu_p(({\rm ord}(\partial_1), c))+1$ different homotopy types of looped gauge groups over $M$ at prime $p$.
\end{theorem}
From this theorem, we see how the fundamental group of $M$ affects its gauge groups, and the situation here is quite different from that in the study of $\mathcal{G}_k(S^4)$ in the literature. In particular, after looping the number of homotopy types of gauge groups over $M$ is much less than that over $S^4$, where the latter is governed by $(k, {\rm ord}(\partial_1))$. Then there is an interesting question that whether this is right without looping. On the other hand, we can work out concrete examples for various groups $G$ listed in Table \ref{tables4guageorder}. In particular, we will see when the looped gauge groups have only one homotopy type, while the general cases can be similarly worked out but will be omitted for the reason of simplicity.

\begin{corollary}[Corollary \ref{mooregaugetrivial2}]\label{corollary2.5}
Let $(M, G, i)$ as in Theorem \ref{theorem2.5}.
If $(G, p, c)$ is in one of the following groups:
\begin{itemize}
\item $G=SU(n)$, $n\leq (p-1)^2+1$, $p\geq 3$, and $\nu_p((n(n^2-1), c))=1$;
\item $G=Sp(n)$, $4\leq 2n\leq (p-1)^2+1$, $p\geq 3$ and $\nu_p((n(2n+1), c))=1$;
\item $G=Spin(2n+1)$, $4\leq 2n\leq (p-1)^2+1$, $p\geq 3$ and $\nu_p((n(2n+1), c))=1$;
\item $G=Spin(2n)$, $6\leq 2n\leq (p-1)^2+1$, $p\geq 5$ and $\nu_p(((n-1)(2n-1), c))=1$;
\item $G=G_2$, $p\geq 3$, $(3\cdot 7)\nmid  c$;
\item $G=F_4$, $p\geq 5$, $(5\cdot 13)\nmid c$;
\item $G=E_6$, $p\geq 5$, $(5\cdot 7\cdot13)\nmid c$;
\item $G=E_7$, $p\geq 7$, $(7\cdot 11\cdot 19)\nmid c$;
\item $G=E_8$, $p\geq 7$, $(7\cdot 11\cdot 13\cdot 19\cdot 31)\nmid c$,
\end{itemize}
then for any $k\in \mathbb{Z}/c$
\[
\Omega^i\mathcal{G}_k(M)\simeq_{(p)} \Omega^i \mathcal{G}_0(M).
\]
\end{corollary}

Let us turn to some applications of Theorem \ref{theorem1}, \ref{theorem2}  and \ref{theorem2.5}. The first application is about the homotopy exponent problem. For any pointed space $Y$, its $p$-th \textit{homotopy exponent} is the least power of $p$ which annihilates the $p$-torsion in $\pi_\ast(X)$, and may be denoted by ${\rm exp}_p(X)$ or simply ${\rm exp}(X)$. Since when $p=2$ the gauge groups in our theorems will be trivial due to the assumptions on $c$, we may always suppose $p$ is an odd prime.

The homotopy exponents of the Lie group $G$ play a crucial role here. Here we are mainly interested in low rank Lie groups in the range of Theriault \cite{Theriault07} (Table \ref{tablelie2intro}). 
\begin{table}[H]
\centering
\caption{Low rank Lie groups in the range of Theriault}
\begin{tabular}{l|p{3.7cm}lp{3.7cm}}
\hline 
$SU(n)$            & $n-1\leq (p-1)(p-2)$        \\ \hline
$Sp(n)$            & $2n\leq (p-1)(p-2)$          \\ \hline
$Spin(2n+1)$            & $2n\leq (p-1)(p-2)$         \\ \hline
$Spin(2n)$            & $2(n-1)\leq (p-1)(p-2)$         \\ \hline
$G_2, F_4, E_6$            & $p\geq5$       \\ \hline
$E_7, E_8$            & $p\geq7$       \\ \hline
\end{tabular}
\label{tablelie2intro}
\end{table}
\noindent First it is well known that rationally 
\[
G\simeq_{(0)} S^{2n_1+1}\times S^{2n_2+1}\times \cdots \times S^{2n_l+1},
\]
where the index set $\mathfrak{t}(X)=\{n_1, n_2,\ldots, n_l\}$ ($n_1\leq n_2\leq \ldots\leq n_l$) is called the \textit{type} of $G$. Le us denote $l(G)= n_l$. Theriault \cite{Theriault07} proved that in his range certain power maps of $\Omega G$ factors multiplicatively through the product of spheres with the same type of $G$ 
\[
\xymatrix{
\Omega G\ar[r]^{p^{r(G)}}  \ar[d]^{}   & \Omega G\ar@{=}[d]\\
\prod\limits_{k\in \mathfrak{t}(G)} \Omega S^{2k+1} \ar[r]^{\ \ \ \lambda}  & \Omega G,
}
\]
which particularly implies that
\[{\rm exp}(G) \leq p^{r(G)+l(G)}.\]
Here the numbers $r(G)$ are acessible. For instance, $r(SU(n))=\nu_p((n-1)!)$ and the values for other classical Lie groups can be obtained from that of $SU(n)$ \cite{Theriault07, Harris61}, while the homotopy exponents of exceptional Lie groups are known, as summarised in Theorem $1.10$ of \cite{DT08}. We can now state one of our estimations of the homotopy exponents of the gauge groups in the range of Theriault. Assume $p\geq 5$ for the reason of simplicity.
\begin{theorem}[Proposition \ref{expformu3}]\label{theoremexpintro}
Let $M$ be a five dimensional oriented closed manifold with $\pi_{1}(M)\cong \mathbb{Z}/c$ ($c>1$) and $H_{2}(M; \mathbb{Z})$ is torsion free. 
Let $G$ be a Lie group in the range of Theriault (Table \ref{tablelie2intro}). Then we have for any $k\in\mathbb{Z}/c$
\[
{\rm exp}(\mathcal{G}_k(M))\leq p^{r(G)+\nu_p({\rm ord}(\partial_1))}\cdot {\rm max}\{p^{r(G)+l(G)},  p^{\nu_p(c)}\}.
\]
\end{theorem}

By the above theorem, it is not hard to compute the explicit upper bounds of the homotopy exponents. To emphasis the structure groups under consideration, let us denote $\mathcal{G}^{G}_k(M)=\mathcal{G}_k(M)$. We state the estimations for the matrix groups here, while the ones for the exceptional Lie groups are listed in Table \ref{tableexp2intro} (at the end of Section \ref{exponentsec}, Page $31$) which is lengthy.
\begin{corollary}\label{coroexpintro1}
Let $M$ as in Theorem \ref{theoremexpintro}. 
For the matrix groups in the range of Theriault (Table \ref{tablelie2intro}) and any $k\in \mathbb{Z}/c$, we have
\begin{eqnarray*}
{\rm exp}(\mathcal{G}^{SU(n)}_k(M))&\leq& {\rm max}(n+2p-5, \nu_p(c)+p-1),\\
{\rm exp}(\mathcal{G}^{Sp(n)}_k(M))&\leq& {\rm max}(2n+2p-6, \nu_p(c)+p-2),\\
{\rm exp}(\mathcal{G}^{Spin(2n+1)}_k(M))&\leq& {\rm max}(2n+2p-6, \nu_p(c)+p-2),\\
{\rm exp}(\mathcal{G}^{Spin(2n)}_k(M))&\leq& {\rm max}(2n+2p-8, \nu_p(c)+p-2).\\
\end{eqnarray*}
\end{corollary}

As the second application of Theorem \ref{theorem1} and \ref{theorem2}, we investigate the periodic phenomena of the homotopy groups of gauge groups derived from the classical Bott periodicity. There are natural inclusions $\mathcal{G}^{SU(n)}_k(M)
\hookrightarrow \mathcal{G}^{SU(n+1)}_k(M)$ and $\mathcal{G}^{Spin(n)}_k(M)
\hookrightarrow \mathcal{G}^{Spin(n+1)}_k(M)$, and we may denote 
\[\mathcal{G}^{SU}_k(M)={\rm colim}_n\mathcal{G}^{SU(n)}_k(M), \  \  \mathcal{G}^{Spin}_k(M)={\rm colim}_n\mathcal{G}^{Spin(n)}_k(M).\]
\begin{theorem}[Section \ref{Bottsection}]\label{theorem3}
Let $M$ be a five dimensional oriented closed manifold with $\pi_{1}(M)\cong \mathbb{Z}/c$ ($c\neq 0$) and $H_{2}(M; \mathbb{Z})$ is torsion free of rank $m-1$. Then after localization away from $c$ we have for any $k\in \mathbb{Z}/c$
\begin{itemize}
\item when $M$ is spin, or non-spin but with further localization away from $2$,
\[\pi_r(\mathcal{G}^{SU}_k(M))\cong \oplus_{m}\mathbb{Z}   \ \  (r\geq 1);\]
\item when $M$ is spin and $r\geq 2$,
\[
\pi_r(\mathcal{G}^{Spin}_k(M))=\left\{\begin{array}{ll}
\oplus_{m-1} \mathbb{Z}\oplus\mathbb{Z}/2  &r\equiv 0 ,1 ,4 ~{\rm mod}~8,\\
\mathbb{Z} & r\equiv 2 ~{\rm mod}~8,\\
\mathbb{Z}\oplus\mathbb{Z}/2  & r\equiv 3 ~{\rm mod}~8,\\
\oplus_{m-1} (\mathbb{Z}\oplus\mathbb{Z}/2)  &r\equiv 5 ~{\rm mod}~8,\\
\mathbb{Z}\oplus\oplus_{2m-2} \mathbb{Z}/2 &r\equiv 6 ~{\rm mod}~8,\\
\mathbb{Z}\oplus\oplus_{m-1} \mathbb{Z}/2  &r\equiv 7~{\rm mod}~8;\\
\end{array}
\right.
\]
\item when $M$ is non-spin with further localization away from $2$, and $r\geq 2$,
\[
\pi_r(\mathcal{G}^{Spin}_k(M))=\left\{\begin{array}{ll}
\oplus_{m-1} \mathbb{Z}&r\equiv 0 ,1 ~{\rm mod}~4,\\
\mathbb{Z} & r\equiv 2, 3 ~{\rm mod}~4.\\
\end{array}
\right.
\]
\end{itemize}
\end{theorem} 

As we remarked earlier, the methods and techniques in this paper can be applied to other situations. Besides of moving to $6$, $7$, or higher dimensional manifolds, we can compute the rational homotopy of the gauge groups easily. Indeed, F\'{e}lix and Oprea \cite{FO09} proved that the rational homotopy type of gauge group of any Lie type principal bundle is the same as that of gauge group of trivial bundle over finite complex. Their proof was to use the structure on the Lie group side. In contrast, with the methods for proving Theorem \ref{theorem1} and \ref{theorem2} working on the base complex side, we can generalize their results for any complex of finite type and general topological group.

Our paper is organized as follows. 
In Section \ref{PB+Gaugesec}, we classify principal bundles over our $5$-manifolds $M$, and prove a general homotopy decomposition of looped gauge groups.
In Section \ref{1decomsec}, we study the homotopy type of $2$-fold suspensions of $M$ when $2\nmid c$, and then show a preliminary homotopy decomposition of looped gauge groups.
In Section \ref{6cdecomsec}, based on the calculation in Section \ref{1decomsec}, we study the homotopy type of $3$-fold suspensions of $M$ when $6\nmid c$, and then combine some arguments in manifold topology to show our first main result Theorem \ref{theorem1}.
In Section \ref{Mtrivialsec}, we weaken the condition on $c$ with the extra cost on the structure of $M$ to show Theorem \ref{theorem2}, and the arguments indeed valid for general Poincar\'{e} duality complex.
In Section \ref{gaugemooresec}, we study the gauge groups over Moore spaces and then prove Theorem \ref{theorem2.5} and Corollary \ref{corollary2.5}.
In Section \ref{exponentsec}, we study the homotopy exponent problem of the gauge groups over $5$-manifolds, and prove Theorem \ref{theoremexpintro} and other results on the primary exponents.
Section \ref{Bottsection} is devoted to the periodic phenomena of the homotopy groups of gauge groups, and Theorem \ref{theorem3} is proved there.
We add an appendix (Section \ref{Qdecomsec}) to discuss the rational homotopy of gauge groups and moduli spaces of connections in general context. We show various formulas of homotopy and cohomology groups of these objects. 

We remark that in the whole paper we may denote the $0$-component of loop space $\Omega^i_0 Z$ just by $\Omega^i Z$ for the reason of simplicity.

\section{Principal bundles over five manifolds and their gauge groups}\label{PB+Gaugesec}
\noindent In this section we develop some basic facts of our $5$-manifolds. In particular, we classify their $G$-principal bundles for certain types of $G$, and also develop a general homotopy decomposition of associated looped gauge groups.

Suppose $M$ is a five dimensional oriented closed manifold with $\pi_{1}(M)\cong \mathbb{Z}/c$ and $H_{2}(M; \mathbb{Z})$ is torsion free of rank $m-1$. By Poincar\'{e} duality and the universal coefficient theorem we have 
\[
H_\ast(M; \mathbb{Z})=\left\{\begin{array}{ll}
\mathbb{Z}/c &\ast=1,\\
\oplus_{i=1}^{m-1}\mathbb{Z} & \ast=2,\\
\oplus_{i=1}^{m-1}\mathbb{Z}\oplus \mathbb{Z}/c  & \ast=3,\\
\mathbb{Z}  & \ast=0, 5\\
0&\hbox{otherwise}.
\end{array}
\right.
\]

Before we go further, let us make some comments on the cell structures of $M$. Firstly, it is well known that closed topological manifolds of dimension other than four are homeomorphic to $CW$ complexes \cite{KS69}. Secondly, 
since we have to consider non-simply connected spaces here, the classical notion of skeleton of a $CW$-complex is not a homotopic functor, that is, we can not talk about the skeletons of non-simply connected spaces in homotopy category. However, since there is 
always a minimal cell structure for any simply connected $CW$-complex $X$ (see Section $4.C$ of \cite{Hatcher02}), we can define $X_i$ to be the $i$-th skeleton of any minimal cell model of $X$ which as a functor is indeed homotopic (see Lemma $2.3$ of \cite{HW18} for details).
\begin{lemma}\label{MBG}
Let $G$ be a simply connected compact simple Lie group with $\pi_4(G)=0$. Then 
\[[M, BG]\cong \pi_{1}(M)\cong \mathbb{Z}/c.\]
\end{lemma}
\begin{proof}
First recall that for any simply connected compact simple Lie group $G$, we have $\pi_2(G)=0$ and $\pi_3(G)\cong \mathbb{Z}$.
Let us choose any cellular model of $M$. Since $\pi_1(M)\cong \mathbb{Z}/c$, there exists a map $P^2(c)\rightarrow M$ such that $H_\ast(M/P^2(c))\cong H_\ast(M)$ for $\ast\geq 2$, and we also have the exact sequence 
\[0=[P^3(c), BG]\rightarrow [M/P^2(c), BG]\rightarrow [M, BG]\rightarrow [P^2(c), BG]=0.\]
Hence the cellular chain complex of $M/P^2(c)$
\[0\rightarrow C_5 \rightarrow C_4 \rightarrow C_3\rightarrow C_2\rightarrow 0,\]
is isomorphic to the chain complex 
\begin{equation}\label{cellSigmaM}
0\rightarrow\mathbb{Z} \stackrel{0}{\rightarrow} \mathbb{Z} \stackrel{c}{\rightarrow}  \oplus_{i=1}^{m}\mathbb{Z} \stackrel{0}{\rightarrow} \oplus_{i=1}^{m-1}\mathbb{Z}\rightarrow 0
\end{equation}
with the morphism $c$ sending the generator to $c$ times of some primitive element (defined in the last part of the introduction) in $\oplus_{i=1}^{m}\mathbb{Z}$, and $[M/P^2(c), BG]\cong [M, BG]$. In the remaining proof let us denote $M_i={\rm sk}_i(M/P^2(c))$ for $i\geq 2$.

We start with the calculation of $[M_4, BG]$. Since there are cofibre sequences
\[\Sigma M_2\rightarrow \Sigma M_3\rightarrow \Sigma(M_3/M_2)\simeq \vee_{i=1}^{m} S^4, \ \ 
M_3\rightarrow M_4 \rightarrow M_4/M_3\simeq S^4 \rightarrow \Sigma M_3,\]
we have a diagram of groups of homotopy classes 
\[
 \xymatrix{
\lbrack\Sigma(M_3/M_2), BG\rbrack \ar[dr]^{c^\ast} \ar[d] \\
\lbrack\Sigma M_3, BG\rbrack \ar[r] \ar[d]  &  \lbrack S^4, BG\rbrack\cong \mathbb{Z} \ar[r] & \lbrack M_4, BG\rbrack \ar[r] & \lbrack M_3, BG\rbrack=0\\
\lbrack\Sigma M_2, BG\rbrack=0,
}
\]
where the row and column are exact and $ c^\ast=c$ by the description of the cellular chain complex of $M/P^2(c)$. Hence $[M_4, BG]\cong \mathbb{Z}/c$.

On the other hand, the cofibre sequence
\[S^4\rightarrow M_4 \rightarrow M_5\rightarrow S^5\]
implies a commutative diagram
\[
 \xymatrix{
&&\lbrack M_4/M_3\simeq S^4, BG \rbrack \ar@{->>}[d] \ar[dr]^{0}\\
\lbrack S^5, BG\rbrack  \ar[r] & \lbrack M_5, BG\rbrack  \ar[r] & \lbrack M_4, BG\rbrack  \ar[r] &\lbrack S^4, BG\rbrack ,
}
\]
where the row is exact and the $0$ morphism is described by the structure of the cellular chain complex of $M/P^2(c)$.
Since by assumption $\pi_4(G)=0$ we see that $[M_5, BG]\cong [M_4, BG]$. Combining the earlier calculation of $[M_4, BG]$ we have proved the lemma.
\end{proof}

We now prove a general homotopy decomposition of looped gauge groups, both which and its analogies will be used often in the rest of the paper.

\begin{proposition}[cf. Proposition $2.1$ of \cite{Theriault10}, Lemma $2.3$ of \cite{So16} and Proposition $2.4$ of \cite{Huang18}]\label{Gsplit}
Let $M$ be a connected manifold, and $G$ be a simply compact simple Lie group. Suppose there exists a CW-complex $Y$ with 
\[ [Y, BG]=0, ~{\rm and}~ {\rm a}~ {\rm map}~ \phi: Y\rightarrow M\] 
such that $\Sigma^{i+1} \phi$ ($i\geq 0$) admits a left homotopy inverse. Then 
\[\Sigma^{i+1} M\simeq \Sigma^{i+1} Y\vee\Sigma^{i+1} X,\]
where $X$ is the homotopy cofibre of $\phi$, and
there is a homotopy equivalence
\[\Omega^{i}\mathcal{G}_\alpha(M)\simeq \Omega^{i}\mathcal{G}_{\alpha^\prime}(X)\times \Omega^{i}{\rm Map}^{\ast}_0(Y, G),\]
where $\alpha\in [M, BG]$, $\alpha^\prime$ is determined by $\alpha$ (and in particular $\alpha^\prime=\alpha$ when $i=0$; see the proof) and ${\rm Map}^{\ast}_0(Y, G)$ is the component of the based mapping spaces ${\rm Map}^{\ast}(Y, G)$ containing the basepoint.
\end{proposition}
\begin{proof}
The homotopy cofibre sequence 
\[Y\stackrel{\phi}{\rightarrow}  M \stackrel{q}{\rightarrow} X \]
induces the the following exact sequence
\[[Y, BG]\leftarrow[M, BG]\stackrel{q^\ast}{\leftarrow}[X, BG].\]
Since $[Y, BG]=0$ by assumption we have $q^\ast$ is surjective. We then choose $\alpha^\prime \in [X, BG]$ such that $q^\ast(\alpha^\prime)=\alpha$.
Then by considering the natural homotopy fibre sequence
\[
{\rm Map}^{\ast}_0(Z, G)\rightarrow  \mathcal{G}_\beta(Z)\rightarrow G\rightarrow {\rm Map}^{\ast}_\beta(Z, BG)\rightarrow {\rm Map}_\beta(Z, BG)\stackrel{{\rm ev}}{\rightarrow} BG
\]
for any class $\beta\in [Z, BG]$, we can form the following homotopy commutative diagram
\[
 \xymatrix{
 \ast  \ar[r]  \ar[d]  & \Omega^{i+1}{\rm Map}^{\ast}_\alpha(M, BG) \ar@{=}[r] \ar[d]  & {\rm Map}^{\ast}_\alpha(\Sigma^{i+1} M, BG)  \ar[d]_{\Sigma^{i+1} \phi}  \\
 \Omega^{i}\mathcal{G}_{\alpha^\prime}(X) \ar[r]  \ar@{=}[d]    & \Omega^{i}\mathcal{G}_\alpha(M)  \ar[r] \ar[d] & {\rm Map}^{\ast}_0(\Sigma^{i+1} Y, BG)  \ar[d] \ar@/_1pc/@{.>}[u]^{\iota} \\
  \Omega^{i}\mathcal{G}_{\alpha^\prime}(X) \ar[r]  \ar[d]    & \Omega^{i} G \ar[r] \ar[d] & {\rm Map}^{\ast}_{\alpha^\prime}(\Sigma^{i} X, BG) \ar[d]\\
  \ast  \ar[r]   &   \Omega^{i}{\rm Map}^{\ast}_\alpha(M, BG) \ar@{=}[r]   &  {\rm Map}^{\ast}_\alpha(\Sigma^{i} M, BG) }
\]
where the rows and columns are homotopy fibre sequence, and the map $\iota$ is a right homotopy inverse. Hence, the fibre sequence in the second row splits and the lemma follows.
\end{proof}

\section{Homotopy decompositions of double-looped gauge groups over $M$}\label{1decomsec}
\noindent In the spirit of Proposition \ref{Gsplit}, we need to study the homotopy structure of $\Sigma^3 M$. Through that we may prove a preliminary homotopy decomposition of looped gauge groups over $M$ when $c$ is odd. 

To simplify notation, for our five dimensional manifold $M$ which is non-simply connected, we may denote the $(i+t)$-th skeleton of $\Sigma^t M$ ($t\geq 1$) to be $\Sigma^t M_i$.
\subsection{Suspension homotopy type of $5$-manifolds}\label{Sigma2Msection}
Let $M$ be a five dimensional oriented closed manifold with $\pi_{1}(M)\cong \mathbb{Z}/c$ ($c\geq 3$ an odd number) and $H_{2}(M; \mathbb{Z})$ is torsion free of rank $m-1$. We want to study the homotopy type of $\Sigma M_4$ and $\Sigma^2 M_4$. Since $\Sigma M_4$ is simply connected, by Theorem $4H.3$ of \cite{Hatcher02} $\Sigma M_4$ has a homology decomposition built from Moore spaces, and then we have a homotopy cofibre sequence
\begin{equation}\label{hdecomSM4}
P^4(c)\vee\bigvee_{j=1}^{m-1} S_j^3\stackrel{h}{\longrightarrow} P^3(c)\vee\bigvee_{i=1}^{m-1}S_i^3 \longrightarrow  \Sigma M_4,
\end{equation}
such that $H_3(h)=0$ and $P^k(c)=S^{k-1}\cup_c e_k$ with an attaching map of degree $c$.

In order to study the homotopy type of $\Sigma M_4$ and its suspension, we need to calculate the possible homotopy classes of maps between Moore spaces. First, let us recall the following version of the Blakers-Massey theorem \cite{BM52}:
\begin{lemma}\label{BMthm}
Given a diagram 
\[
\xymatrix{
 F \ar[rd]^{i}  &&\\
 A \ar@{.>}[u]^{g} \ar[r] & X \ar[r]^{f} &Y
}
\]
where the second row is a cofibre sequence, $F$ is the homotopy fibre of $f$, $A$ is $k$-connected and $i$ is $m$-connected (i.e., $\pi_\ast(i)$ is an isomorphism when $\ast<m$ and an epimorphism when $\ast=m$), then 
the induced map $g: A \rightarrow F$ is $(k+m)$-connected, and we have a long exact sequence of homotopy groups
\[\pi_{k+m}(A)\rightarrow \pi_{k+m}(X)\rightarrow \pi_{k+m}(Y) \rightarrow \pi_{k+m-1}(A)\rightarrow \cdots \rightarrow \pi_{k+1}(X)\rightarrow \pi_{k+1}(Y) \rightarrow 0.\]
\hfill $\Box$
\end{lemma}

Now let us determine the homotopy groups $\pi_n(P^n(c))$.
\begin{lemma}\label{pi3p3c}
Let $2\nmid c$.
\[
\pi_3(P^3(c))\cong \mathbb{Z}/c, \ \ \  \pi_n(P^n(c))\cong 0 \ \ (n\geq 4)
\]
\end{lemma}
\begin{proof}
Let us first compute $\pi_n(P^n(c))$ and consider the cofibre sequence 
\[S^{n-1}\rightarrow P^n(c)\rightarrow S^n,\]
which is a fibre sequence up to degree $2n-3$ by Lemma \ref{BMthm}. Then since $n\geq 4$ we have an exact sequence 
\[\pi_{n}(S^{n-1})\rightarrow \pi_{n}(P^n(c))\rightarrow \pi_n(S^n)\]
where $\pi_{n}(S^{n-1})\cong \mathbb{Z}/2$, $\pi_n(S^n)\cong\mathbb{Z}$. However $\pi_{n}(P^n(c))$ is a torsion group without elements of order $2$ since $P^n(c)$ is contractible after localization away from $2$. Hence it has to be $\pi_{n}(P^n(c))=0$. 

It remains to prove $\pi_3(P^3(c))\cong \mathbb{Z}/c$.
Consider the fibration 
\begin{equation}\label{FPK2}
F\longrightarrow P^3(c)\stackrel{\kappa}{\longrightarrow} K(\mathbb{Z}/c, 2)
\end{equation}
determined by the generator in $H^2(P^3(c);\mathbb{Z}/c)\cong [P^3(c), K(\mathbb{Z}/c, 2)]\cong \mathbb{Z}/c$. This is the bottom stage of the Postnikov system of $P^3(c)$ and then $F$ is $2$-connected and $\pi_3(P^3(c))\cong \pi_3(F)\cong H_3(F;\mathbb{Z})$. For the latter homology group, let us consider the primary case first, and we need some information of the homology of $K(\mathbb{Z}/c, 2)$ (e.g., see the proof of Theorem $11.7$ of \cite{Neisendorfer10}):
\[
\bar{H}_\ast(K(\mathbb{Z}/p^r, 2); \mathbb{Z})=\left\{\begin{array}{ll}
\mathbb{Z}/p^r &\ast=2, 4,\\
0 & \ast=0, 1, 3, 5,\\
\end{array}
\right.
\]
for $p\geq3$.
Then by expecting the Serre spectral sequence of (\ref{FPK2}), we see that 
\[
\pi_3(P^3(p^r))\cong  H_3(F;\mathbb{Z})\cong\mathbb{Z}/p^r.
\]
For the general case, we know that $P^3(c)$ is a torsion space at prime factors of $c$. Hence we have $\pi_3(P^3(c))\cong \mathbb{Z}/c$.
\end{proof}

In order to compute the groups of homotopy classes of maps from Moore spaces, we also need to apply the theory of homotopy groups with coefficients (\cite{CMN79, Neisendorfer, Neisendorfer10}). The $n$-th homotopy group of a space $X$ with coefficients in $\mathbb{Z}/c$ is defined to be
\[\pi_{n}(X;\mathbb{Z}/c):=[P^n(c), X].\] 
There is a universal coefficient theorem relating the homotopy groups with coefficients to the usual ones.
\begin{lemma}[Theorem $1.3.1$ of \cite{Neisendorfer}]\label{UCT}
For a space $X$ and $n\geq 2$, there is a natural exact sequence
\[0\rightarrow \pi_{n}(X)\otimes \mathbb{Z}/c \rightarrow \pi_{n}(X;\mathbb{Z}/c) \rightarrow {\rm Tor}^{\mathbb{Z}}(\pi_{n-1}(X), \mathbb{Z}/c)\rightarrow 0.\]
\hfill $\Box$
\end{lemma}
Now we use the above universal coefficient theorem to determine some homotopy groups with coefficients.
\begin{lemma}\label{pi4s3c}
Let $2\nmid c$.
\[
\pi_{4}(S^3;\mathbb{Z}/c)=\pi_{5}(S^4;\mathbb{Z}/c)=0.
\]
\end{lemma}
\begin{proof}
By Lemma \ref{UCT}, there is a short exact sequence
\[0\rightarrow \pi_{n+1}(S^n)\otimes \mathbb{Z}/c \rightarrow \pi_{n+1}(S^n;\mathbb{Z}/c) \rightarrow {\rm Tor}^{\mathbb{Z}}(\pi_{n}(S^n), \mathbb{Z}/c)\rightarrow 0.\]
Since ${\rm Tor}^{\mathbb{Z}}(\pi_{n}(S^n), \mathbb{Z}/c)=0$, $\pi_{n+1}(S^n;\mathbb{Z}/c)\cong \pi_{n+1}(S^n)\otimes \mathbb{Z}/c$. Since $\pi_{n+1}(S^n)\cong \mathbb{Z}/2$ for $n\geq 3$, the lemma follows.
\end{proof}

\begin{lemma}\label{pi4p3cc}
Let $2\nmid c$.
\[
\pi_{4}(P^3(c); \mathbb{Z}/c)\cong \mathbb{Z}/c, \ \ \  \pi_{5}(P^4(c); \mathbb{Z}/c)=0.
\]
\end{lemma}
\begin{proof}
By Lemma \ref{UCT}, there is a natural exact sequence
\[0\rightarrow \pi_{n+1}(P^n(c))\otimes \mathbb{Z}/c \rightarrow \pi_{n+1}(P^n(c); \mathbb{Z}/c) \rightarrow {\rm Tor}^{\mathbb{Z}}(\pi_{n}(P^n(c)), \mathbb{Z}/c)\rightarrow 0.\]

When $n=3$, $\pi_{4}(P^3(c))=0$ by Lemma $3.3$ of \cite{So16} and $\pi_{3}(P^3(c))\cong \mathbb{Z}/c$ by Lemma \ref{pi3p3c}. Hence $\pi_{4}(P^3(c); \mathbb{Z}/c)\cong \mathbb{Z}/c$.

When $n=4$, by Lemma \ref{BMthm}, the cofibre sequence 
\[S^3\rightarrow P^4(c)\rightarrow S^4\]
is a fibre sequence up to degree $5$. Hence we have an exact sequence 
\[\mathbb{Z}/2\cong\pi_{5}(S^3)\rightarrow \pi_5(P^4(c))\rightarrow \pi_5(S^4)\cong \mathbb{Z}/2.\]
Since $\pi_5(P^4(c))$ can only have odd torsions we see $\pi_5(P^4(c))=0$. Then 
\[\pi_{5}(P^4(c); \mathbb{Z}/c)\cong  {\rm Tor}^{\mathbb{Z}}(\pi_{4}(P^4(c)), \mathbb{Z}/c)=0\] by Lemma \ref{pi3p3c}.
\end{proof}

By the previous computations of homotopy groups, we see that $\Sigma M_4$ in general is not a bouquet of Moore spaces while $\Sigma^2 M_4$ indeed is, i.e.,  
\begin{equation}\label{hdecomS2M4}
\Sigma^2 M_4\simeq P^6(c)\vee P^4(c)\vee \bigvee_{i=1}^{m-1}(S^5\vee S^4).
\end{equation}

We now investigate the homotopy type of $\Sigma^3 M$ and introduce the follow lemma analogous to Corollary $4.3$ of \cite{Huang18}:

\begin{lemma}[cf. Lemma $4.1$ and Corollary $4.3$ of \cite{Huang18} and Lemma $2.5$ of \cite{So16}]\label{Sigmasplit}
Let $X$ be a stable $CW$-complex with cell structure 
\[X\simeq \bigvee_{i=1}^{m} S^n\cup_f e^{n+k},\]
i.e., the attaching map $f\in \pi_{n+k-1}(S^n)\cong \pi_{k-1}(\mathbb{S})$ ($k>1$) is in the stable range. Suppose \[\pi_{k-1}(\mathbb{S})\cong \mathbb{Z}/d_1\oplus \mathbb{Z}/d_2\oplus \cdots\oplus \mathbb{Z}/d_r, \]
such that $d_i|d_{i+1}$ for $1\leq i\leq r$. Then if $m\geq r$ we have a homotopy equivalence 
\[X\simeq  Z\vee \bigvee_{i=1}^{m-r} S^{n},\]
where $Z$ is the cofibre of the inclusion $\bigvee_{i=1}^{m-r} S^{n}\hookrightarrow X$. \hfill $\Box$
\end{lemma}
With the above lemma, we can split off spheres from $\Sigma^3 M$.
\begin{proposition}\label{decomSigma3M}
Let $M$ be a five dimensional oriented closed manifold with $\pi_{1}(M)\cong \mathbb{Z}/c$ ($c\geq 3$ an odd number) and $H_{2}(M; \mathbb{Z})$ is torsion free of rank $m-1$. Then if $m\geq 2$ we have a homotopy equivalence
\[
\Sigma^3 M \simeq \Sigma Z \vee \bigvee_{i=1}^{m-2}(S^6\vee S^5),
\]
where $Z$ is the cofibre of the composition $\bigvee_{i=1}^{m-2}(S^5\vee S^4)\hookrightarrow \Sigma^2 M_4\hookrightarrow \Sigma^2 M$.
\end{proposition}
\begin{proof}
By (\ref{hdecomS2M4}), we have 
\[
\Sigma^2 M \simeq \Sigma^2 M_4 \cup e^{7} \simeq \big(P^6(c)\vee P^4(c)\vee \bigvee_{i=1}^{m-1}(S^5\vee S^4)\big)\cup_{f} e^{7}.
\]
Then in $\Sigma^3 M$ the attaching map of the top cell $\Sigma f$ lies in 
\[\pi_{7}(P^7(c))\oplus\pi_{7}(P^5(c))\oplus\oplus_{i=1}^{m-1}(\pi_{7}(S^6)\oplus\pi_{7}(S^5)),\]
where the last two factors are stable homotopy groups and $\pi_{1}(\mathbb{S})\cong\pi_{2}(\mathbb{S})\cong \mathbb{Z}/2$. Hence by Lemma \ref{Sigmasplit} the lemma follows. 
\end{proof}

\subsection{Homotopy decompositions of $\Omega^2 \mathcal{G}_k(M)$}
\begin{theorem}\label{1stdecomgauge}
Let $M$ be a five dimensional oriented closed manifold with $\pi_{1}(M)\cong \mathbb{Z}/c$ ($c\geq 3$ an odd number) and $H_{2}(M; \mathbb{Z})$ is torsion free of rank $m-1$.
Let $G$ be a simply connected compact simple Lie group with $\pi_4(G)=0$. Then we have the homotopy equivalence 
\[\Omega^2 \mathcal{G}_k(M)\simeq \Omega^2\mathcal{G}_k(M^\prime)\times \prod_{i=1}^{m-2} \Omega^4 G,\]
where $k\in \mathbb{Z}/c$ and $M^\prime$ is the homotopy cofibre of some map $\phi: \bigvee_{i=1}^{m-2}S^2\hookrightarrow M$ (see the proof). 
\end{theorem}
\begin{proof}
Let us first specify the map in the statement of the theorem.
Since $\pi_{1}(M)\cong\mathbb{Z}/c$, the inclusion of the first skeleton $S^1\hookrightarrow M$ can be extended to a map from $P^2(c)$ and we have the cofibre sequence \[P^2(c)\rightarrow M \stackrel{q}{\rightarrow} K.\] Then $K$ is simply connected and by calculating the homology it is easy to see that ${\rm sk}_2 K\simeq \bigvee_{i=1}^{m-1}S^2$. Also by expecting the universal covering of $M$, we see $\pi_2(M)\rightarrow \pi_{2}(K)$ is surjective. Furthermore $\pi_2(K)\cong \pi_2(\bigvee_{i=1}^{m-1}S^2)$ by expecting the cellular chain complex of $K$ (\ref{cellSigmaM}).
Hence there is a map $\psi: P^2(c)\vee \bigvee_{i=1}^{m-1} S^2\rightarrow M$ which induces an isomorphism of homology up to degree $2$. In particular, by choosing any $m-2$ of the $n-1$ two-spheres we get the map $\phi: \bigvee_{i=1}^{m-2}S^2\hookrightarrow M$.

Now we turn to study $\mathcal{G}_k(M)$, the name of which is justified by the fact $[M, BG]\cong \mathbb{Z}/c$ by Lemma \ref{MBG}. Let $Y=\bigvee_{i=1}^{m-1}S^2$ and $i=2$. Then $[Y, BG]=0$ and $\Sigma^3\phi$ admits a left homotopy inverse by Lemma \ref{decomSigma3M}. Hence by Proposition \ref{Gsplit} we have the homotopy equivalence 
\[\Omega^2 \mathcal{G}_k(M)\simeq \Omega^2\mathcal{G}_k(M^\prime)\times \Omega^2{\rm Map}_0^{\ast}(\bigvee_{i=1}^{m-2}S^2, G),\]
and $[M^\prime, BG]=[M, BG]=\mathbb{Z}/c$. The theorem then follows.
\end{proof}

\section{Further decomposition of double-looped gauge groups when $6\nmid c$}\label{6cdecomsec}
\noindent In this section, suppose $M$ as before is a five dimensional closed manifold with $H_{2}(M; \mathbb{Z})$ is torsion free of rank $m-1$ but $\pi_{1}(M)\cong \mathbb{Z}/c$ ($6\nmid c$). In this case we will develop further homotopy decomposition of $\Omega^2 \mathcal{G}_k(M)$, and in particular prove Theorem \ref{theorem1}.
For that purpose, we need to study the top attaching map of $\Sigma^3 M$.

\subsection{The homotopy aspect of $\Sigma^3 M$ when $6\nmid c$}
Recall that in Section \ref{Sigma2Msection}, we have proved that (\ref{hdecomS2M4})
\[
\Sigma^2 M_4\simeq P^6(c)\vee P^4(c)\vee \bigvee_{i=1}^{m-1}(S^5\vee S^4).
\]
In order to determine the top attaching map, we need to calculate some homotopy groups and related suspension images. The follow lemma shows that the component of the top attaching map of $\Sigma^3M$ on $P^5(c)$ is trivial when $6\nmid c$, and this is important information for us to decompose the gauge groups further.
\begin{lemma}\label{pi7p5c3c}
If $2\nmid c$, then 
\[\pi_{6}(P^4(c))\cong\pi_{7}(P^5(c))\cong  \mathbb{Z}/c\oplus \mathbb{Z}/(3,c),\]
and the image of the suspension homomorphism 
\[\Sigma:\pi_{6}(P^4(c))\rightarrow \pi_{7}(P^5(c)) \]
is $\mathbb{Z}/(3,c)$. 
In particular, ${\rm Im}(\Sigma)=0$ when $6\nmid c$.
\end{lemma}
\begin{proof}
We will obtain the lemma once we prove Lemma \ref{pi6p4c}, Lemma \ref{pi7p5c} and Lemma \ref{pi6p4cSigma}.
\end{proof}

Let us first recall some useful lemmas.
\begin{lemma}\label{pinchsplit}
Let $F^{n+1}\{c\}$ be the homotopy fibre of the pinch map $P^{n+1}(c)\rightarrow S^{n+1}$ ($n\geq 2$). Then there is a principal fibre sequence
\[\Omega S^{n+1} \stackrel{s_{n+1}}{\longrightarrow} F^{n+1}\{c\} \longrightarrow P^{n+1}(c),\] 
and the reduced homology 
\[ \bar{H}_\ast(F^{n+1}\{c\})\cong \oplus_{k=1}^{\infty}\mathbb{Z}\{y_k\} (|y_k|=kn)\] 
is a free $H_\ast(\Omega S^{n+1})$-module on $y_1$ with $\iota_n^{k-1}\cdot y_1= y_k$.

Moreover, the $(3n-1)$-skeleton ${\rm sk}_{3n-1}(F^{n+1}\{c\})$ is the homotopy cofibre of $c$-times of the Whitehead product $c\omega_n: S^{2n-1}\rightarrow S^n$. In particular, when $c=2^r$
\[{\rm sk}_{6n-4}(F^{2n}\{2^r\})\simeq S^{2n-1}\vee S^{4n-2}.\]
\end{lemma}
\begin{proof}
The module structure of the homology of $F^{n+1}\{c\}$ is proved in Proposition $8.1$ and Corollary $8.2$ of \cite{CMN79}, and the proof of the remaining part of the lemma is the same as that of Proposition $4.15$ of \cite{Wu2003} once we replace $P^{n+1}(2)$ by $P^{n+1}(c)$ there.
\end{proof}

The following version of $EHP$-sequence is a consequence of the Blakers-Massey theorem (Lemma \ref{BMthm}):
\begin{lemma}[Theorem $12.2.2$ of \cite{Whitehead78}]\label{EHP}
Let $W$ be an $(n-1)$-connected space. There is a long exact sequence 
\begin{eqnarray*}
&\ \ & \ \ \ \pi_{3n-2}(W)\rightarrow \pi_{3n-1}(\Sigma W)\rightarrow \pi_{3n-1}(\Sigma W\wedge W)\rightarrow \cdots \\
&&\cdots \rightarrow\pi_{q}(W)\stackrel{E}{\rightarrow}\pi_{q+1}(\Sigma W)\stackrel{H}{\rightarrow}\pi_{q+1}(\Sigma W\wedge W)\stackrel{P}{\rightarrow}\pi_{q-1}(W)\rightarrow \cdots
\end{eqnarray*}
\hfill $\Box$
\end{lemma}

With these computational tools, we can now prove Lemma \ref{pi6p4c}, Lemma \ref{pi7p5c} and Lemma \ref{pi6p4cSigma} successively.

\begin{lemma}\label{pi6p4c}
If $2\nmid c$, then 
$\pi_{6}(P^4(c))\cong\mathbb{Z}/c\oplus \mathbb{Z}/(3,c)$.
\end{lemma}
\begin{proof}
By Lemma \ref{pinchsplit} we have ${\rm sk}_8F^{4}\{c\}\simeq S^3\cup_{c\omega_3}e^6$. However $\omega_3\in\pi_5(S^3)$ is $0$ since $S^3$ is a Lie group and in particular an $H$-space. Hence $F^{4}\{c\}\simeq S^3\vee S^6$ and we have a homotopy commutative diagram
\[
\xymatrix{
S^3 \ar[d]_{E} \ar[r]^{c} & S^3 \ar@{_{(}->}[d] \ar[rd] \\
\Omega S^4 \ar[r]^{s_4} & S^3\vee S^6 \ar[r] & P^4(c)\ar[r] & S^4,
}
\]
where the second row is a homotopy fibre sequence up to degree $7$. Then we have a commutative diagram of homotopy groups
\[
\xymatrix{
\pi_{6}(S^3) \ar[d]_{E} \ar[r]^{c_\ast=c} & \pi_{6}(S^3) \ar@{_{(}->}[d] \ar[rd] \\
\pi_{6}(\Omega S^4) \ar[r]^{s_{4\ast}} & \pi_{6}(S^3\vee S^6) \ar[r] & \pi_{6}(P^4(c))\ar[r] & \pi_{6}(S^4),
}
\]
where the second row is exact, $c_\ast=c$ for $S^3$ is a group, and we only need to consider the localization of the diagram away from $2$ since 
$P^4(c)$ is contractible after localization at $2$. Under this context we have \cite{Toda62}
\[\pi_{6}(S^4)\cong 0, \  \ \pi_{6}(S^3)\cong \mathbb{Z}/3, \  \ \pi_{6}(\Omega S^4)\cong \pi_{6}(\Omega S^7) \oplus \pi_6(S^3)\cong \mathbb{Z}\oplus \mathbb{Z}/3,\]
and $E$ is an injection. It follows that the image of $s_{4\ast}$ on $\mathbb{Z}/3$-summand is $\mathbb{Z}/(3,c)$. For the free part, notice that by Lemma \ref{pinchsplit} $H_\ast(s_{4\ast})$ is an $H_\ast(\Omega S^4)$-module homomorphism. In particular, 
\[H_6(s_{4\ast}): H_6(\Omega S^4)\cong H_{6}(\Omega S^7)\rightarrow H_6(S^3\vee S^6)\]
is a degree $c$ morphism. Hence, by Hurewicz theorem $s_{4\ast}$ on $\mathbb{Z}$ is a morphism of degree $c$. Combining the two parts we see that $\pi_{6}(P^4(c))\cong\mathbb{Z}/c\oplus \mathbb{Z}/(3,c)$ by the exactness of the second row of the diagram.
\end{proof}

\begin{lemma}\label{pi7p5c} 
If $2\nmid c$, then 
$\pi_{7}(P^5(c))\cong  \mathbb{Z}/c\oplus \mathbb{Z}/(3,c)$.
\end{lemma}
\begin{proof}
By Lemma \ref{pinchsplit} we have ${\rm sk}_{11}F^{5}\{c\}\simeq S^4\cup_{c\omega_4}e^8$. We need to compute $\pi_7(S^4\cup_{c\omega_4}e^8)$ first. By Lemma \ref{BMthm} we have a commutative diagram of homotopy groups 
\[
\xymatrix{
\pi_7(S^7) \ar[rd]^{c\omega_4} \ar[d]_{\cong} \\
\pi_7(\Omega S^8)\ar[r] & \pi_7(S^4) \ar[r]^{i_\ast\ \ \ } & \pi_7(S^4\cup_{c\omega_4}e^8) \ar[r] & \pi_7(S^8)=0,
}
\]
where the second row is exact and $\pi_7(S^4)=\{\nu_4\}\oplus\{E\nu^\prime\}\cong \mathbb{Z}\oplus \mathbb{Z}/12$. Also we know that 
\[\Sigma: \pi_7(S^4)=\{\nu_4\}\oplus\{E\nu^\prime\}\rightarrow \pi_8(S^5)=\{\nu_5\}\]
sending $\nu_4$ to $\nu_5$ and $E\nu^\prime$ to $2\nu_5$. Hence by the Hopf invariant one \cite{Adams60} and $E\omega_4=0$ it is easy to see $\omega=2\nu_4-E\nu^\prime$. Now let us localize everything away from $2$. We then get $\omega=\nu_4-E\nu^\prime\in \mathbb{Z}\oplus \mathbb{Z}/3$. Hence by the above diagram we get  $\pi_7(S^4\cup_{c\omega_4}e^8)\cong \mathbb{Z}/c\oplus \mathbb{Z}/3$.

Now we turn to study the follwoing homotopy commutative diagram
\[
\xymatrix{
S^4 \ar[d]_{E} \ar[r]^{c} & S^4 \ar@{_{(}->}[d]_{i} \ar[rd] \\
\Omega S^5 \ar[r]^{s_5} & S^4\cup_{c\omega_4}e^8 \ar[r] & P^5(c)\ar[r] & S^5,
}
\]
where the second row is a homotopy fibre sequence up to degree $10$. We then get a commutative diagram of homotopy groups localized away from $2$
\[
\xymatrix{
\pi_{7}(S^4) \ar@{->>}[d]_{E} \ar[r]^{c_\ast} & \pi_{7}(S^4) \ar@{->>}[d]_{i_\ast} \ar[rd] \\
\pi_{7}(\Omega S^5) \ar[r]^{s_{5\ast} \ \ \ } & \pi_{7}(S^4\cup_{c\omega_4}e^8) \ar[r] & \pi_{7}(P^5(c))\ar[r] & \pi_{7}(S^5)=0,
}
\]
where the second row is exact. Again by the Hopf invariant one see that $c_\ast$ is a morphism of degree $c^2$ on the $\mathbb{Z}$ summand and of degree $c$ on the $\mathbb{Z}/3$ summand. With this information in hand we can show that ${\rm Im}(s_{5\ast})={\rm Im}(s_{5\ast}E)={\rm Im}(i_\ast c_\ast)$ is $0$ if $3\mid c$ and $\mathbb{Z}/3$ if $3\nmid c$. Hence by exactness we get $\pi_{7}(P^5(c))\cong  \mathbb{Z}/c\oplus \mathbb{Z}/(3,c)$.
\end{proof}

\begin{lemma}\label{pi6p4cSigma}
If $2\nmid c$, then 
the image of the suspension homomorphism 
\[\Sigma:\pi_{6}(P^4(c))\rightarrow \pi_{7}(P^5(c)) \]
is $\mathbb{Z}/(3,c)$. 
\end{lemma}
\begin{proof}
Applying Lemma \ref{EHP} we get an exact sequence 
\[
\pi_6(P^4(c))\stackrel{E}{\rightarrow}\pi_7(P^5(c))
\stackrel{H}{\rightarrow}\pi_7(\Sigma P^4(c)\wedge P^4(c))\stackrel{P}{\rightarrow} \pi_5(P^4(c)),
\]
where $\pi_5(P^4(c))=0$ by the proof of Lemma \ref{pi4p3cc}. On the other hand, we have homotopy decomposition (see Proposition $6.2.2$ of \cite{Neisendorfer})
\[P^m(c)\wedge P^n(c)\simeq P^{m+n}(c)\vee P^{m+n-1}(c).\]
Hence, 
\[\pi_{7}(\Sigma P^4(c)\wedge P^4(c))\cong\pi_{7}(P^9(c)\vee P^8(c))\cong\pi_{7}(P^8(c))\cong \mathbb{Z}/c.\]
By Lemma \ref{pi6p4c} and Lemma \ref{pi7p5c}, 
\[\pi_6(P^4(c))\cong \pi_7(P^5(c))\cong \mathbb{Z}/c\oplus \mathbb{Z}/(3,c).\]
Hence when $3\nmid c$ the homomorphism $H$ has to be an isomorphism and then $E$ is trivial. 

For the remaining case when $3\mid c$, by checking the proof of Lemma \ref{pi6p4c} and Lemma \ref{pi7p5c} we notice that the $\mathbb{Z}/3$-summands of $\pi_6(P^4(c))$ and $\pi_7(P^5(c))$ both come from the homotopy groups of sphere, i.e., we have a commutative diagram 
\[
\xymatrix{
\pi_6(S^3) \ar[r]^{E}\ar[d] &\pi_7(S^4)\ar[d]\\
\pi_6(P^4(c))\ar[r]^{E}&\pi_7(P^5(c)).
}
\]
Since $H$ is an epimorphism, the $\mathbb{Z}/c$-summand of $\pi_7(P^5(c))$ is not in the image of $E$. Hence ${\rm Im}(E)\cong\mathbb{Z}/3$ and we complete the proof.
\end{proof}

Based on the computations, we can now split off Moore spaces from $\Sigma^3 M$.
\begin{proposition}\label{decomSigma3M6c}
Let $M$ be a five dimensional oriented closed manifold with $\pi_{1}(M)\cong \mathbb{Z}/c$ ($6\nmid c$) and $H_{2}(M; \mathbb{Z})$ is torsion free of rank $m-1$. Then if $m\geq 2$ we have a homotopy equivalence
\[
\Sigma^3 M \simeq \Sigma Z^\prime \vee P^5(c)\vee P^7(c)\vee \bigvee_{i=1}^{m-2}(S^6\vee S^5),
\]
where $Z^\prime$ is the cofibre of the composition $P^4(c)\vee P^6(c)\vee\bigvee_{i=1}^{m-2}(S^5\vee S^4)\hookrightarrow \Sigma^2 M_4\hookrightarrow \Sigma^2 M$.
\end{proposition}
\begin{proof}
The proof is similar to that of Proposition \ref{decomSigma3M}.
By (\ref{hdecomS2M4}), we have 
\[
\Sigma^2 M \simeq \Sigma^2 M_4 \cup e^{7} \simeq \big(P^6(c)\vee P^4(c)\vee \bigvee_{i=1}^{m-1}(S^5\vee S^4)\big)\cup_{f} e^{7}.
\]
Then the lemma follows from Lemma \ref{Sigmasplit} and Lemma \ref{pi7p5c3c} and Lemma \ref{pi3p3c}.
\end{proof}

\subsection{Homotopy decompositions of $\Omega^2 \mathcal{G}_k(M)$ when $6\nmid c$}\label{sectiongauge6c}
We are now going to produce the homotopy decompositions of looped gauge groups by repeated application of  Proposition \ref{Gsplit}.
As the definition of homotopy groups with coefficients, let us define 
\[\Omega^n(X;c)={\rm Map}_0^{\ast}(P^n(c), X).\]
Let $M$ be a five dimensional oriented closed manifold with $\pi_{1}(M)\cong \mathbb{Z}/c$ ($6\nmid c$) and $H_{2}(M; \mathbb{Z})$ is torsion free of rank $m-1$. Let $G$ be a simply connected compact simple Lie group with $\pi_4(G)=0$. 
\begin{lemma}\label{gaugedecom6c1}
We have the homotopy equivalence 
\[\Omega^2 \mathcal{G}_k(M)\simeq \Omega^2\mathcal{G}_k(M^{\prime\prime})\times \Omega^4(G;c)\times \prod_{i=1}^{m-2} \Omega^4 G,\]
where $k\in \mathbb{Z}/c$ and $M^{\prime\prime}$ is the homotopy cofibre of the map $\psi: P^2(c)\vee\bigvee_{i=1}^{m-2}S^2\hookrightarrow M$ (constructed in the proof of Theorem \ref{1stdecomgauge}). 
\end{lemma}
\begin{proof}
We have constructed $\psi: P^2(c)\vee\bigvee_{i=1}^{m-2}S^2\hookrightarrow M$ in the proof of Theorem \ref{1stdecomgauge}. Then let $Y=P^2(c)\vee\bigvee_{i=1}^{m-2}S^2$. We see that $[Y, BG]=0$ and the cofibre $M^{\prime\prime}$ satisfies $[M^{\prime\prime}, BG]\cong [M, BG]\cong \mathbb{Z}/c$. By Proposition \ref{decomSigma3M6c}, $\Sigma^3\psi$ admits a left homotopy inverse. Then the lemma follows from Proposition \ref{Gsplit}. 
\end{proof}

\begin{lemma}\label{gaugeMprimeprime}
Let $M^{\prime\prime}$ be the complex in Lemma \ref{gaugedecom6c1}. We have homotopy decompositions of gauge groups
\[\mathcal{G}_k(M^{\prime\prime})\simeq \mathcal{G}_k(L) \times \prod_{i=1}^{m-2} \Omega^3 G,\]
where $k\in \mathbb{Z}/c$ and $L$ is the homotopy cofibre of the map $\psi^\prime: \bigvee_{i=1}^{m-2}S^3\hookrightarrow M^{\prime\prime}$ (constructed in the proof). 
\end{lemma}
\begin{proof}
By Lemma \ref{gaugedecom6c1}, we know $M^{\prime\prime}$ is simply connected and its homology satisfies
\[H_{2}(M^{\prime\prime})\cong \mathbb{Z}, \ \ \ H_{\ast}(M^{\prime\prime})\cong H_\ast(M), \ \  ~{\rm for}~\ast\geq 3.\]
Hence the fourth skeleton $M^{\prime\prime}_4$ of $M^{\prime\prime}$ is determined by the cofibre sequence 
\[P^3(c)\vee\bigvee_{i=1}^{m-1}S^2\stackrel{h}{\rightarrow} S^2\rightarrow M^{\prime\prime}_4,\]
where $H_2(h)=0$. Then we have $M^{\prime\prime}_4\simeq S^2\cup C(P^3(c))\vee \bigvee_{i=1}^{m-1}S^3$ where $C(-)$ is denoted to be the cone of space, and define $\psi^\prime$ as the composition
\[\bigvee_{i=1}^{m-2}S^3\hookrightarrow M^{\prime\prime}_4 \hookrightarrow M^{\prime\prime}.\]
Further since $\pi_4(S^3; \mathbb{Z}/c)=0$ by Lemma \ref{pi4s3c}, $\Sigma M^{\prime\prime}_4$ splits as 
\[\Sigma M^{\prime\prime}_4\simeq S^2\vee P^4(c)\vee \bigvee_{i=1}^{m-1}S^3,\]
and $\Sigma M^{\prime\prime}\simeq \big(S^2\vee P^4(c)\vee \bigvee_{i=1}^{m-1}S^3\big)\cup e^6$.
Then by Lemma \ref{Sigmasplit} we see that 
\[\Sigma M^{\prime\prime}\simeq  \bigvee_{i=1}^{m-2}S^3\vee\big(S^2\vee P^4(c)\vee S^3\big)\cup e^6,\]
which implies $\Sigma \psi^\prime$ admits a left homotopy inverse. Since $[\bigvee_{i=1}^{m-2}S^3, BG]=0$ and it is easy to check $[L, BG]\cong [M^{\prime\prime},BG]$, we apply Proposition \ref{Gsplit} to prove the lemma. 
\end{proof}

\begin{lemma}\label{gaugeL}
Let $L$ be the complex in Lemma \ref{gaugeMprimeprime}. We have homotopy decompositions of gauge groups
\[\mathcal{G}_k(L)\simeq \mathcal{G}_k(P^4(c)) \times {\rm Map}_0^\ast(T, G),\]
where $k\in \mathbb{Z}/c$ and the complex $T$ satisfies and is characterized by $T\simeq \Sigma \mathbb{C}P^2\vee S^2$ if the Steenrod operation $Sq^2$ is non-trivial on $H^\ast(L;\mathbb{Z}/2)$, or $T\simeq S^3\vee S^2\vee S^5$ or $S^3\vee S^2\cup_{\eta^2}e^5$ otherwise.
\end{lemma}
\begin{proof}
By Lemma \ref{gaugeMprimeprime}, 
\[L\simeq \big(S^3\vee (S^2\cup C(P^3(c)))\big)\cup e^5\ \ ~{\rm and}~ \ \ \Sigma L\simeq (S^4\vee S^3 \vee P^5(c))\cup e^6.\]
Since $\pi_n(P^n(c))=0$ by Lemma \ref{pi3p3c}, we have
\[\Sigma L\simeq(S^4\vee S^3)\cup e^6 \vee P^5(c),\]
and a homotopy commutative diagram of cofibre sequences  
\[
\xymatrix{
S^4  \ar[r] \ar[d]  & S^3\vee (S^2\cup C(P^3(c))) \ar[r] \ar[d] & L \ar[d]_{q} \\
\ast  \ar[r] \ar[d]  &  P^4(c) \ar@{=}[r] \ar[d]                         & P^4(c) \ar[d]\\
S^5 \ar[r]  & S^4\vee S^3 \ar[r]  & (S^4\vee S^3)\cup e^6 , 
}
\]
which define the map $q$. Hence $[L, BG]\cong [P^4(c), BG]\cong \mathbb{Z}/c$ and $\Sigma L$ splits through $\Sigma q$. Moreover since $\pi_{n+2}(S^{n})=\{\eta_n^2\}\cong \mathbb{Z}/2$ and $\pi_{n+2}(S^{n+1})=\{\eta_n\}\cong \mathbb{Z}/2$ for $n\geq 2$, we see that $(S^4\vee S^3)\cup e^6$ can be desuspended to be $(S^3\vee S^2)\cup_\alpha e^5$ with the attaching map $\alpha=(a\eta_3, b\eta_2^2)^{T}$ as column vector. 
we can constructed a homotopy commutative diagram (page $97$ of \cite{Harper02})
\[
\xymatrix{
S^4 \ar[r]^{\alpha} \ar@{=}[d] & S^3\vee S^2 \ar[r] \ar[d]^{A}& (S^3\vee S^2)\cup e^5 \ar[d]_{\simeq}\\
S^4 \ar[r]^{\alpha^\prime} & S^3\vee S^2 \ar[r] & T,
}
\]
where 
\[
A=
\begin{pmatrix}
  1 & 0\\
  \eta_2 & 1
\end{pmatrix}
, \ \ \ {\rm and}~ \ \ 
A\alpha =
\begin{pmatrix}
  1 & 0\\
  \eta_2 & 1
\end{pmatrix}
\begin{pmatrix}
  a\eta_3 \\
  b\eta_2^2
\end{pmatrix}
=
\begin{pmatrix}
  a\eta_3 \\
  (a+b)\eta_2^2
\end{pmatrix}
=\alpha^\prime.
\]
In particular, we see the homotopy types of $T$ can only be of three types: $S^2\vee S^3\vee S^5$, $S^2\vee \Sigma \mathbb{C}P^2$ or $S^3\vee S^2\cup_{\eta_2^2}e^5$.

We then form the commutative diagram analogous to that in the proof of Proposition \ref{Gsplit}
\[
 \xymatrix{
 \ast  \ar[r]  \ar[d]  & \Omega{\rm Map}^{\ast}_k(L, BG) \ar@{=}[r] \ar[d]  & {\rm Map}^{\ast}_k(\Sigma L, BG)  \ar[d]  \\
 \mathcal{G}_k(P^4(c)) \ar[r]  \ar@{=}[d]    & \mathcal{G}_k(L)  \ar[r] \ar[d] & {\rm Map}^{\ast}_0(\Sigma T, BG)  \ar[d] \ar@/_1pc/@{.>}[u] \\
  \mathcal{G}_k(P^4(c)) \ar[r]  \ar[d]    & G \ar[r] \ar[d] & {\rm Map}^{\ast}_k(P^4(c), BG) \ar[d]\\
  \ast  \ar[r]   &   {\rm Map}^{\ast}_k(L, BG) \ar@{=}[r]   &  {\rm Map}^{\ast}_k(L, BG),  
}
\]
and similarly the lemma follows.
\end{proof}
So far by Lemma \ref{gaugedecom6c1}, Lemma \ref{gaugeMprimeprime} and Lemma \ref{gaugeL}, we notice that we have reduced the looped gauge group $\Omega^2\mathcal{G}_k(M)$ to that of Moore spaces $\Omega^2\mathcal{G}_k(P^4(c))$ and various mapping spaces including the mapping space ${\rm Map}_0^\ast(T, G)$. As specified in Lemma \ref{gaugeL}, the homotopy type of $T$ has only $3$ possibilities. However, using the manifold structure of $M$ we can show that one of the candidates $S^3 \vee S^2\cup_{\eta^2}e^5$ is impossible to be our $T$.
\begin{lemma}\label{HtypeT}
Let $T$ be the complex in Lemma \ref{gaugeL} constructed through Lemma \ref{gaugedecom6c1}, Lemma \ref{gaugeMprimeprime} and Lemma \ref{gaugeL} from the $5$-dimensional closed manifold $M$. Then 
\[T\not\simeq S^3 \vee S^2\cup_{\eta^2}e^5.\]
\end{lemma}
\begin{proof}
The proof is similar to and based on that of an example of Gitler and Stasheff (see also page $32$ of \cite{MM79}), which shows that a four cell $5$-dimensional Poincar\'{e} duality complex is not homotopy equivalent to a manifold. Roughly speaking, they constructed a $3$-cell Poincar\'{e} duality complex 
\[C=(S^3\vee S^2)\cup_f e^5,\]
where $f$ is the sum of the Whitehead product $[\iota_3, \iota_2]$ and the map $\iota_2\circ\eta^2$. 
Suppose that $C$ is a manifold, then $C$ is spin with trivial Stiefel-Whitney classes by the Wu formula (\cite{Wu55}, Page $132$ of \cite{MS75}). 
By Atiyah duality, the Spanier-Whitehead dual of $\Sigma^{\infty}_{+}C$ is the Thom spectrum of the spin stable normal bundle of the form 
\[\mathbb{F}=\mathbb{S}^0\cup_{\eta^2}e^3\vee \mathbb{S}^2\vee \mathbb{S}^5.\]
Then there is a map $\mathbb{F}\rightarrow  MSpin$ which is an isomorphism on $\pi_0$. Note that $H^{i}(MSpin;\mathbb{Z}/2)$ vanishes for $1\leq i\leq 3$. However, there is a secondary operation $\Phi$ detecting $\eta^2$ (Page 95-97 of \cite{Harper02}), that is, $\Phi: H^{0}(\mathbb{F};\mathbb{Z}/2)\rightarrow H^{3}(\mathbb{F};\mathbb{Z}/2)$ is non-trivial. Since $\Phi$ is natural in this situation, we then get a contradiction.

Return to our situation. Suppose $T\simeq S^3 \vee S^2\cup_{\eta^2}e^5$. Then on our manifold $M$ the Steenrod operations are trivial by Proposition \ref{decomSigma3M6c}, which implies that the Stiefel-Whitney classes of $M$ vanish by the Wu formula. In particular, $M$ is spin. 
Then by using the similar proof as sketched above we can show that there are no secondary operations in $H^\ast(M; \mathbb{Z}/2)$ detecting $\eta^2$. By our procedure to produce the complex $T$ we see this also holds for $T$ and we get a contradiction. Hence the lemma follows. 
\end{proof}

Now by combining all of the results above, we can show our main theorem in this section.
\begin{theorem}\label{gaugedecom6c}
Let $M$ be a five dimensional oriented closed manifold with $\pi_{1}(M)\cong \mathbb{Z}/c$ ($6\nmid c$) and $H_{2}(M; \mathbb{Z})$ is torsion free of rank $m-1$. Let $G$ be a simply connected compact simple Lie group with $\pi_4(G)=0$. We have the following homotopy equivalences:
\begin{itemize}
\item if $M$ is a spin manifold,
\[\Omega^2 \mathcal{G}_k(M)\simeq \Omega^2\mathcal{G}_k(P^4(c))\times \Omega^4(G;c)\times \Omega^7G\times \prod_{i=1}^{m-1} (\Omega^4 G  \times \Omega^5 G );\]
\item if $M$ is a non-spin manifold,
\[\Omega^2 \mathcal{G}_k(M)\simeq \Omega^2\mathcal{G}_k(P^4(c))\times \Omega^3 {\rm Map}_0^\ast(\mathbb{C}P^2, G)\times \Omega^4(G;c)\times \prod_{i=1}^{m-1} \Omega^4 G  \times \prod_{i=1}^{m-2} \Omega^5 G,\]
\end{itemize}
where $k\in \mathbb{Z}/c$.
\end{theorem}
\begin{proof}
Combining Lemma \ref{gaugedecom6c1}, Lemma \ref{gaugeMprimeprime}, Lemma \ref{gaugeL} and Lemma \ref{HtypeT}, we see that $\Omega^2 \mathcal{G}_k(M)$ decomposes into the indicated forms according to the two homotopy types of $T$ in Lemma \ref{gaugeL}. By simple calculation using the Wu formula, we see that the non-triviality of the second Stiefel-Whitney class $\omega_2$ is equivalent to the existence of a non-trivial $Sq^2$-action on $M$. 
Hence the two decompositions indeed correspond to the possible spin and non-spin structure of $M$ respectively.
\end{proof}

\section{The case when $M$ is stably parallelizable}\label{Mtrivialsec}
\noindent As we may notice, many results in the previous sections also valid for general complexes with described cohomology groups. In this section, let us start within a little more general context.
Let $M$ be a five dimensional Poincar\'{e} duality space with $\pi_{1}(M)\cong \mathbb{Z}/c$ ($c\geq 3$ an odd number) and $H_{2}(M; \mathbb{Z})$ is torsion free of rank $m-1$. Recall that for any Poincar\'{e} duality space $X$ there is a stably unique \textit{Spivak normal fibration} $\xi$ \cite{Spivak67, Browder72} which is a spherical fibration $\xi\rightarrow X$ such that the Thom complex $X^{\xi}$ of $\xi$ admits a map 
\[S^N\rightarrow X^{\xi}\]
sending the fundamental class $[S^N]\in H_{N}(S^N; \mathbb{Z})$ to the image of fundamental class $[X]$ under the Thom isomorphism.  
\begin{lemma}\label{Sigma4M}
Suppose $M$ has precisely one $5$-dimensional cell. Then if the Spivak normal fibration $\xi$ of $M$ is trivial, we have  
\[
\Sigma^4 M \simeq S^9\vee P^8(c)\vee P^6(c)\vee \bigvee_{i=1}^{m-1}(S^7\vee S^6).
\]
\end{lemma}
\begin{proof}
By (\ref{hdecomS2M4}), we have 
\[
\Sigma^2 M \simeq \Sigma^2 M_4 \cup e^{7} \simeq \big(P^6(c)\vee P^4(c)\vee \bigvee_{i=1}^{m-1}(S^5\vee S^4)\big)\cup_{f} e^{7}.
\]
On the other hand, in Lemma $3.8$ of \cite{KLPT17} it was proved that the attaching map for the top cell of any Poincar\'{e} duality complex with single top cell is stably null homotopic if and only if its Spivak normal fibration is trivial. In particular $f$ is stably trivial and the lemma then follows by Freudenthal suspension theorem.
\end{proof}
Based on the nice stably splittable structure of $M$, we can produce a complete homotopy decomposition of looped gauge groups in this situation.
\begin{theorem}\label{gaugedecomPoincare}
Let $M$ be a five dimensional Poincar\'{e} duality space with precisely one $5$-dimensional cell, $\pi_{1}(M)\cong \mathbb{Z}/c$ ($c\geq 3$ an odd number) and $H_{2}(M; \mathbb{Z})$ is torsion free of rank $m-1$. Let $G$ be a simply connected compact simple Lie group with $\pi_4(G)=0$. Then if the Spivak normal fibration $\xi$ of $M$ is trivial, we have the homotopy equivalence
\[\Omega^3 \mathcal{G}_k(M)\simeq \Omega^3\mathcal{G}_k(P^4(c))\times \Omega^5(G;c)\times \Omega^8G\times \prod_{i=1}^{m-1} (\Omega^5 G  \times \Omega^6 G),\]
where $k\in \mathbb{Z}/c$.
\end{theorem}
\begin{proof}
The proof is similar to that of Theorem \ref{gaugedecom6c}, and since we have proved that $\Sigma^4M$ is completely decomposable by Lemma \ref{Sigma4M} we need to apply Proposition \ref{Gsplit} for $i=3$ and several choices of $X$. Indeed we have the cofibre sequences
\begin{eqnarray*}
Y_1=P^2(c)\vee\bigvee_{i=1}^{m-1}S^2&\rightarrow& M\rightarrow X_1\simeq \big(P^4(c)\vee\bigvee_{i=1}^{m-1}S^3\big)\cup e^5,\\
Y_2=\bigvee_{i=1}^{m-1}S^3&\rightarrow& X_1\rightarrow X_2\simeq P^4(c)\vee S^5,\\
Y_3=S^5&\rightarrow& X_2\rightarrow P^4(c),
\end{eqnarray*}
where we use the fact $\pi_4(P^4(c))=0$ by Lemma \ref{pi3p3c}.
Then by applying Proposition $\ref{Gsplit}$ to $X_i$ ($i=1, 2, 3$) successively we get the decomposition stated in the theorem.
\end{proof}

Now we specify to the case when $M$ is an oriented closed manifold.
\begin{definition}[Section $3$ of \cite{KM63}]\label{stablepara}
A manifold $M$ is called \textit{stably parallelizable} if the Whitney sum $\tau(M)\oplus \epsilon^1$ is a trivial bundle, where $\tau(M)$ is the tangent bundle of $M$ and $\epsilon^k$ denote a trivial $k$-plane bundle over $M$.
\end{definition}

\begin{remark}
\begin{itemize}
\item[1.] Given $r\geq 0$, Kervaire and Milnor \cite{KM63} showed that $\xi\oplus \epsilon^r$ is trivial if and only if $\xi$ is trivial provided by $k> n$ where $\xi$ is a $k$-dimensional plane bundle over an $n$-dimensional complex.
\item[2.] Hirsch \cite{Hirsch59} showed that an $n$-manifold M is stably parallelizable if and only if $M$ is orientable and immerses into $\mathbb{R}^{n+1}$; and if only if its normal bundle is stably trivial, and then if only if $M$ is orientable and immerses into $\mathbb{R}^{n+k+1}$ with a transversal $k$-field. 
\item[3.] Barden \cite{Barden65} has shown a closed simply connected $5$-manifold admits an immersion into $\mathbb{R}^{6}$ if and only if its second Stiefel-Whitney class vanishes. On the contrary, for a non-simply connected closed $5$-manifold $M$ there is a possibly non-vanishing obstruction in $H^4(M_4, M_3; \{\pi_3(V_5(\mathbb{R}^6))\})$ where $\{\pi_3(V_5(\mathbb{R}^6))\})$ is a system of local coefficients and $\pi_3(V_5(\mathbb{R}^6))
\cong \pi_3(SO(6))\cong \mathbb{Z}$.
\end{itemize}
\end{remark}

The following corollary follows immediately from the fact that, if a Poincar\'{e} duality space is homotopy equivalent to a manifold via some map, then the spherical fibration associated to the pullback of the normal bundle of that manifold is isomorphic to the Spivak normal fibration.

\begin{corollary}\label{gaugedecompimfld}
Let $M$ be a five dimensional stably parallelizable oriented closed manifold with precisely one $5$-dimensional cell, $\pi_{1}(M)\cong \mathbb{Z}/c$ ($c\geq 3$ an odd number) and $H_{2}(M; \mathbb{Z})$ is torsion free of rank $m-1$. Let $G$ be a simply connected compact simple Lie group with $\pi_4(G)=0$. Then we have 
\[\Omega^3 \mathcal{G}_k(M)\simeq \Omega^3\mathcal{G}_k(P^4(c))\times \Omega^5(G;c)\times \Omega^8G\times \prod_{i=1}^{m-1} (\Omega^5 G  \times \Omega^6 G),\]
where $k\in \mathbb{Z}/c$.
\end{corollary}
\begin{proof}
Since $M$ is stably parallelizable, its normal bundle is then stably trivial. Then the top cell of $M$ splits stably and by Theorem \ref{gaugedecomPoincare} we get the corollary.
\end{proof}

\section{Gauge groups over Moore spaces}\label{gaugemooresec}
\noindent In the previous sections, we have noticed that the gauge groups over Moore spaces are involved in the homotopy decompositions of the gauge groups over $5$-manifolds. Hence, we may focus on $\mathcal{G}_{k}(P^4(c))$ in this section. 

Recall that there exists the natural homotopy fibre sequence
\begin{equation}\label{sepgaugefibre}
{\rm Map}^{\ast}_0(Z, G)\rightarrow  \mathcal{G}_\beta(Z)\rightarrow G\stackrel{\partial_\beta}{\rightarrow} {\rm Map}^{\ast}_\beta(Z, BG)\rightarrow {\rm Map}_\beta(Z, BG)\stackrel{{\rm ev}}{\rightarrow} BG
\end{equation}
for any class $\beta\in [Z, BG]$. The homotopy of the gauge group $\mathcal{G}_\beta(Z)$ is then captured by that of the connecting map $\partial_\beta$.
\begin{lemma}[Lang \cite{Lang73}]\label{langlemma}
If $Z\simeq \Sigma W$, then through the isomorphisms
\[
[G, {\rm Map}^{\ast}_\beta(Z, BG)]\cong [G, {\rm Map}^{\ast}_0(Z, BG)]\cong [W\wedge G, G],
\]
$\partial_\beta$ is identified with the Samelson product $\langle \beta, {\rm id}_G\rangle$.
In particular, 
\[\partial_{k\beta}\simeq k\partial_{\beta}.\]
\end{lemma}
For Moore spaces, we have the usual homotopy cofibre sequence 
\[S^3\stackrel{c}{\rightarrow} S^3\stackrel{f}{\rightarrow} P^4(c)\stackrel{q}{\rightarrow}S^4.\]
The gauge group $\mathcal{G}_k(P^4(c))$ is then related to $\mathcal{G}_{\tilde{k}}(S^4)$ for any $\tilde{k}\in \{k+ci~|~i\in\mathbb{Z}\}$. 
Denote ${\rm Map}_\alpha^{\ast}(P^n(c), X)=\Omega^n_\alpha(X; c)$.
The following lemmas is immediate by the naturality of (\ref{sepgaugefibre}).
\begin{lemma}\label{orderconnectionlemma}
There is a homotopy commutative diagram
\[
\xymatrix{
G \ar[r]^{\tilde{\partial}_{\tilde{k}}}  \ar[d]^{\partial_{k}}     &  \Omega^4_{\tilde{k}}BG \ar[d]^{q^\ast}  \\
\Omega^4_k(BG; c) \ar@{=}[r]                              &  \Omega^4_k(BG; c).
} 
\]
In particular,
\[
{\rm ord}(\partial_k)~|~{\rm ord}(\tilde{\partial}_{\tilde{k}}),
\]
for any $\tilde{k}\in \{k+ci~|~i\in\mathbb{Z}\}$, where ${\rm ord}(\alpha)$ is the least nonnegative integer $m$ such that $m\alpha$ is null-homotopic provided by the existence of a multiplication operation on $\alpha$.
\end{lemma}
There are many studies in the literature on the order of the connecting map $\tilde{\partial}_{\tilde{k}}$, which is crucial for the classification of $\mathcal{G}_k(S^4)$ for various Lie groups. For our purpose (for instance, to make sure that $\pi_4(G)=0$), we summarise part of them in Table \ref{tables4guageorder}, in which $\infty$ represents for the integral cases.
\begin{table}[H]
\centering
\caption{Order of $\tilde{\partial}_1$ for $\mathcal{G}_1(S^4)$}
\begin{tabular}{lp{1.7cm}lp{2.7cm}l}
\hline
$G$          &       $p$               &        ${\rm ord}(\tilde{\partial}_1)$  & reference  \\ \hline
$SU(2)$   &    $p\geq 3$     &        $3$                         & Kono \cite{Kono91}\\ \hline  
$SU(3)$   & $\infty$         &         $24$                        &  Hamanaka-Kono \cite{HK06}\\ \hline
$SU(5)$   & $\infty$         &         $120$                      &  Theriault  \cite{Theriault15}\\ \hline
$SU(n)$    & $n\leq (p-1)^2+1$, $p\geq 3$  & $n(n^2-1)$   &Theriault \cite{Theriault17}\\ \hline
$Sp(n)$    $(n\geq 2)$ & $2n\leq (p-1)^2+1$, $p\geq3$  & $n(2n+1)$  &Kishimoto-Kono \cite{KK18}\\\hline
$Spin(2n+1)$ $(n\geq 2)$  &$2n\leq (p-1)^2+1$, $p\geq 3$ & $n(2n+1)$ & Theriault \cite{Theriault10} and Kishimoto-Kono \cite{KK18} \\ \hline
$Spin(2n)$ $(n\geq 3)$ &$2(n-1)\leq (p-1)^2+1$, $p\geq 5$ & $(n-1)(2n-1)$ &Theriault \cite{Theriault10} and Kishimoto-Kono \cite{KK18}\\ \hline
$G_2$      &$p\geq 3$     &          $3\cdot 7$                        &Kishimoto-Theriault-Tsutaya \cite{KTT17} \\  \hline
$F_4$      &$p\geq 5$      &          $5^2\cdot 13$   \\
$E_6$      &$p\geq 5$      &          $5^2\cdot 7\cdot13$   &  Hasui-Kishimoto-So-\\
$E_7$      &$p\geq 7$      &          $7\cdot 11\cdot 19$ &Theriault \cite{HKST18} \\
$E_8$      &$p\geq 7$      &          $7^2\cdot 11^2\cdot 13\cdot 19\cdot 31$\\ \hline
\end{tabular}
\label{tables4guageorder}
\end{table}

In \cite{Theriault102}, Theriault proved a very useful lemma in the study of the homotopy of gauge groups, which is a general statement and then should be applicable to other situations. We will mainly use ${\rm mod}~p$ version of his lemma. Let $(k, l)$ denote the greatest common divisor of $k$ and $l\in \mathbb{Z}$. If $k$ and $l\in \mathbb{Z}/q$ for some $q\in\mathbb{Z}$, we use $(k, l)$ to denote the set of all the greatest common divisors $(\tilde{k},\tilde{l})$, where $\tilde{k}$ and $\tilde{l}$ are any representatives of $k$ and $l$ respectively. 
\begin{lemma}[Lemma $3.1$ of \cite{Theriault102}]\label{gaugecriteriontherilemma}
Let $X$ be a space and $Y$ be an $H$-space with a homotopy inverse. Suppose that there is a map $f:X\rightarrow Y$ of order $m$, where $m$ is finite. For any integer $k$, let $F_k$ be the homotopy fibre of $k\circ f$.

If $(m,k)=(m,l)$, then $F_k\simeq_{(p)}F_l$ for any prime $p$. In particular, if everything is of $p$-local for some fixed prime $p$, and $(m,k)\cap(m,l)\neq \emptyset$, then $F_k\simeq_{(p)}F_l$.
\end{lemma}

Let $\mathfrak{f}: \mathbb{Z}\rightarrow \mathbb{Z}$ be the function of the numbers of factors of integers, that is, $\mathfrak{f}(m)=(e_1+1)(e_2+1)\ldots(e_r+1)$ if $m=\pm p_1^{e_1}p_2^{e_2}\ldots p_r^{e^r}$ as the decomposition of primes. Let $\nu_p(m)$ be the number of powers of $p$ which divide the integer $m$.

\begin{proposition}\label{mooregaugecriterion}
Let $d=({\rm ord}(\partial_1), c)$. If $(k, d)=(l,d)$, then $\mathcal{G}_k(P^4(c))\simeq_{(p)} \mathcal{G}_l(P^4(c))$ for any prime $p$. In particular, there are at most $\mathfrak{f}(d)$ different homotopy types of gauge groups over $P^4(c)$.

Localize at a prime $p$. If $\nu_p((k, d))=\nu_p((l, d))$, then $\mathcal{G}_k(P^4(c))\simeq_{(p)} \mathcal{G}_l(P^4(c))$. In particular, there are at most $\nu_p(d)+1$ different homotopy types of gauge groups over $P^4(c)$ at prime $p$.
\end{proposition}
\begin{proof}
By applying Lemma \ref{gaugecriteriontherilemma} to the connecting map $\partial_1$ of (\ref{sepgaugefibre}) and the classification of $G$-bundles over $P^4(c)$ by $[P^4(c), BG]\cong \mathbb{Z}/c$, we need to classify the sets ($0\leq k\leq c$) 
\[
S_k=\{({\rm ord}(\partial_1), k+ci)~|~i\in \mathbb{Z}\},
\]
that is, to check necessary condition when $S_k\cap S_l\neq \emptyset$. 
It is easy to see each $S_k$ is equal to $\{({\rm ord}(\partial_1), k+di)~|~i\in \mathbb{Z}\}$.
Now by Dirichlet's theorem on arithmetic progressions, there are infinitely many prime numbers in the arithmetic progression
\[\{\frac{k}{(k,d)}+\frac{d}{(k,d)}i~|~i\in \mathbb{Z}\},\]
provided by $k$, $d\neq0$.
Then picking any large enough prime $\xi$ such that $k+di=(k,d)\xi$ for some $i$, we see that $S_k$ contains $(k,d)$.
Then by Lemma \ref{gaugecriteriontherilemma}, $\mathcal{G}_{l}(P^4(c))\simeq\mathcal{G}_{(k,d)}(P^4(c))$ for any $l$ with $(l,d)=(k,d)$. The proof of the proposition is completed.
\end{proof}

We then can apply Proposition \ref{mooregaugecriterion} to Table \ref{tables4guageorder} via Lemma \ref{orderconnectionlemma}, and obtain a partial classification of the gauge groups over Moore spaces, especially the  upper bounds of the numbers of different homotopy types. Here for the reason of simplicity, we may only state the results for the trivial cases, which already illustrate the role of $c$ for the gauge groups over Moore spaces comparing to that over $S^4$.

\begin{corollary}\label{mooregaugetrivial1}
When $G=SU(2)$, $SU(3)$, and $6\nmid c$, or $G=SU(5)$, and $30\nmid c$, we have
\[
\mathcal{G}_k(P^4(c))\simeq_{(p)} {\rm Map}_0(P^4(c), G)\simeq \Omega^4_0(G; c)\times G,
\]
for any $k\in \mathbb{Z}/c$ and prime $p$.
\end{corollary}

\begin{corollary}\label{mooregaugetrivial2}
If $(G, p, c)$ is in one of the following groups:
\begin{itemize}
\item $G=SU(n)$, $n\leq (p-1)^2+1$, $p\geq 3$, and $\nu_p((n(n^2-1), c))=1$;
\item $G=Sp(n)$, $4\leq 2n\leq (p-1)^2+1$, $p\geq 3$ and $\nu_p((n(2n+1), c))=1$;
\item $G=Spin(2n+1)$, $4\leq 2n\leq (p-1)^2+1$, $p\geq 3$ and $\nu_p((n(2n+1), c))=1$;
\item $G=Spin(2n)$, $6\leq 2n\leq (p-1)^2+1$, $p\geq 5$ and $\nu_p(((n-1)(2n-1), c))=1$;
\item $G=G_2$, $p\geq 3$, $(3\cdot 7)\nmid  c$;
\item $G=F_4$, $p\geq 5$, $(5\cdot 13)\nmid c$;
\item $G=E_6$, $p\geq 5$, $(5\cdot 7\cdot13)\nmid c$;
\item $G=E_7$, $p\geq 7$, $(7\cdot 11\cdot 19)\nmid c$;
\item $G=E_8$, $p\geq 7$, $(7\cdot 11\cdot 13\cdot 19\cdot 31)\nmid c$,
\end{itemize}
then for any $k\in \mathbb{Z}/c$
\[
\mathcal{G}_k(P^4(c))\simeq_{(p)}  \Omega^4_0(G; c)\times G.
\]
\end{corollary}

Now let us return to the gauge groups over our non-simply connected $5$-manifolds. We can then combining our results on $\Omega^i\mathcal{G}_k(M)$ in Section \ref{6cdecomsec} and \ref{Mtrivialsec} to get partial a classification of them.

\begin{theorem}\label{gaugedecom6cclass}
Let $M$ be a five dimensional oriented closed manifold with $\pi_{1}(M)\cong \mathbb{Z}/c$ ($6\nmid c$) and $H_{2}(M; \mathbb{Z})$ is torsion free of rank $m-1$. Let $G$ be a simply connected compact simple Lie group with $\pi_4(G)=0$. 

If $(k, d)=(l,d)$, then $\Omega^2\mathcal{G}_k(M)\simeq_{(p)} \Omega^2\mathcal{G}_l(M)$ for any prime $p$. 
Localize at a prime $p$. If $\nu_p((k, d))=\nu_p((l, d))$, then $\Omega^2\mathcal{G}_k(M)\simeq_{(p)} \Omega^2\mathcal{G}_l(M)$. In particular, if $(p,c)=1$ or $({\rm ord}(\partial_1), c)=1$, 
\[
\Omega^2\mathcal{G}_k(M)\simeq_{(p)} \Omega^2{\rm Map}_0(M, G).
\]
\end{theorem}

\begin{theorem}\label{gaugedecomPoincareclass}
Let $M$ be a five dimensional Poincar\'{e} duality space with precisely one $5$-dimensional cell and trivial Spivak normal fibration, $\pi_{1}(M)\cong \mathbb{Z}/c$ ($c\geq 3$ an odd number) and $H_{2}(M; \mathbb{Z})$ is torsion free of rank $m-1$. Let $G$ be a simply connected compact simple Lie group with $\pi_4(G)=0$.

If $(k, d)=(l,d)$, then $\Omega^3\mathcal{G}_k(M)\simeq_{(p)} \Omega^3\mathcal{G}_l(M)$ for any prime $p$. 
Localize at a prime $p$. If $\nu_p((k, d))=\nu_p((l, d))$, then $\Omega^3\mathcal{G}_k(M)\simeq_{(p)} \Omega^3\mathcal{G}_l(M)$. In particular, if $(p,c)=1$ or $({\rm ord}(\partial_1), c)=1$, 
\[
\Omega^3\mathcal{G}_k(M)\simeq_{(p)} \Omega^3{\rm Map}_0(M, G).
\]
\end{theorem}

With the help of Table \ref{tables4guageorder}, the concrete descriptions of looped gauge groups over $M$ including the analogous results to Proposition \ref{mooregaugetrivial1} and \ref{mooregaugetrivial2} can be worked out easily, but we may omit the detail here.

\section{Applications to the homotopy exponent problem}\label{exponentsec}
\noindent In this section, we apply the results in the previous sections to study the homotopy exponent problem. For any pointed space $Y$, its $p$-th \textit{homotopy exponent} is the least power of $p$ which annihilates the $p$-torsion in $\pi_\ast(X)$, and may be denoted by ${\rm exp}_p(X)$ or simply ${\rm exp}(X)$. We are actually interested in the upper bounds of the odd primary homotopy exponents of gauge groups.

It is sometimes useful to consider the mapping spaces from the Moore spaces as the homotopy fibres of power maps.
\begin{lemma}\label{mappingmooredes}
\[{\rm Map}^\ast(P^{m+1}(c), X) \simeq \Omega^mX\{c\},\]
where $\Omega^mX\{c\}$ is the homotopy fibre of the power map $c:\Omega^m X\rightarrow \Omega X$.
\end{lemma}
\begin{proof}
For the homotopy cofibration
\[
S^m\stackrel{c}{\rightarrow}S^m\rightarrow P^{m+1}(c)
\]
defining $P^{m+1}(c)$, we can apply the cofunctor ${\rm Map}^\ast(-,X)$ to get a homotopy fibration 
\[
{\rm Map}^\ast(P^{m+1}(c), X)\rightarrow \Omega^m X\stackrel{c}{\rightarrow}\Omega^m X.
\]
The lemma then follows.
\end{proof}

The homotopy exponents of $\Omega^mX\{c\}$ are accessible.
\begin{lemma}[Proposition $0.5$ of \cite{Neisen83}]\label{expxclemma}
\[{\rm exp}(\Omega^2X\{c\})\leq p^{\nu_p(c)}.\]
In particular, if $X=G$ is a homotopy associative $H$-space, then 
\[{\rm exp}(\Omega G\{c\})\leq p^{\nu_p(c)}.\]
\hfill $\Box$
\end{lemma}

\begin{lemma}\label{expformu1}
Let $G$ be a simply connected compact simple Lie group with $\pi_4(G)=0$ (let $n\geq 5$ when $G=Spin(n)$). Let $M$ be a five dimensional oriented closed manifold with $\pi_{1}(M)\cong \mathbb{Z}/c$ and $H_{2}(M; \mathbb{Z})$ is torsion free of rank $m-1$. Suppose that either $6\nmid c$, or $2\nmid c$ and the tangent bundle of $M$ is stably trivial. 
Then for any $k\in \mathbb{Z}/c$ we have 
\[
{\rm exp}(\mathcal{G}_k(M))\leq p^{\epsilon(G, p)}\cdot {\rm max}\{{\rm exp}(\mathcal{G}_k(P^4(c)), {\rm exp}(G),  p^{\nu_p(c)}\},
\]
where $\epsilon(G, p)=0$ or $1$ defined as follows.
\begin{itemize}
\item If $p=3$ and $G=Spin(6)$, and $SU(i)$ with $2\leq i\leq 4$, then $\epsilon(G, p)=1$;
\item Otherwise, $\epsilon(G, p)=0$.
\end{itemize}
\end{lemma}
\begin{proof}
Since ${\rm exp}(G\{c\})\leq p^{\nu_p(c)}$ by Lemma \ref{expxclemma} and the fact that $G\{c\}$ is simply connected, we actually need to prove 
\[
{\rm exp}(\mathcal{G}_k(M))\leq p^{\epsilon(G, p)}\cdot {\rm max}\{{\rm exp}(\mathcal{G}_k(P^4(c)), {\rm exp}(G),  {\rm exp}(G\{c\})\}.
\]
Since our theorems in Section \ref{6cdecomsec} and \ref{Mtrivialsec} are about $\Omega^2$- or $\Omega^3$-decompositions of gauge groups, we need to deal with the exponents of the low dimensional homotopy groups first.
From (\ref{sepgaugefibre}), we have the long exact sequence of homotopy groups
\[
\cdots\rightarrow \pi_{\ast+1}(G)\rightarrow \pi_\ast({\rm Map}^\ast_0(M, G))\rightarrow \pi_\ast(\mathcal{G}_k(M))\rightarrow \pi_\ast(G)\rightarrow \cdots,
\]
which implies that $\pi_{i}({\rm Map}^\ast_0(M, G))\cong \pi_{i}(\mathcal{G}_k(M))$ for $i=0$, $1$, $\pi_2({\rm Map}^\ast_0(M, G))\rightarrow \pi_{2}(\mathcal{G}_k(M))$ is an epimorphism, and $\pi_{3}({\rm Map}^\ast_0(M, G))\cong \pi_{3}(\mathcal{G}_k(M))$ modulo possibly one copy of $\mathbb{Z}$. Hence the exponents of $\pi_{i}(\mathcal{G}_k(M)$ are at most that of $\pi_{i}({\rm Map}^\ast_0(M, G))$ for $0\leq i\leq 3$.
It is clear that $\pi_{0}({\rm Map}^\ast_0(M, G))=0$, while $\pi_{i}({\rm Map}^\ast_0(M, G))\cong [\Sigma^{i+1}M, BG]$ for $i\geq 1$. By (\ref{hdecomS2M4}), we have 
\begin{equation}\label{expformu1eq1}
[\Sigma^{i+1}M_4, BG]\cong 
[\Sigma^{i-1}(P^6(c)\vee P^4(c)\vee \bigvee_{i=1}^{m-1}(S^5\vee S^4)), BG].
\end{equation}
On the other hand, associated to the cofibre $S^5\stackrel{f}{\rightarrow}\Sigma M_4\rightarrow \Sigma M$, there is an exact sequence 
\begin{equation}\label{expformu1eq2}
[\Sigma^{i+1}M_4, G]\stackrel{(\Sigma^i f)^\ast}{\rightarrow} [S^{i+5}, G] \rightarrow [\Sigma^i M, G]\rightarrow [\Sigma^i M_4, G]
\end{equation}
for each $i\geq 1$.
An important observation is that $(\Sigma^i f)^\ast$ is trivial when $\pi_{i+5}(G)$ is free abelian. To see this, first notice that $[P^{m+1}(c), BG]$ is a torsion, and by (\ref{expformu1eq1}) we only need to consider the restriction $(\Sigma^i f)^\ast:  [S^{i+4}\vee S^{i+3}, BG]\rightarrow [S^{i+5}, G]$ induced from that map $S^{i+5}\stackrel{\Sigma^i f}{\rightarrow}\Sigma^{i+1}M_4\stackrel{q}{\rightarrow}S^{i+4}\vee S^{i+3}$. However since $\pi_{i+5}(S^{i+4})\cong \pi_{i+5}(S^{i+3})\cong \mathbb{Z}/2$, $(\Sigma^i f)^\ast$ is trivial at our odd prime $p$. Then by exactness of (\ref{expformu1eq2}), the free part of $[S^{i+5}, G]$ contributes nothing to the exponent of $[\Sigma^i M, G]$.

Now by checking the homotopy groups of classical and exceptional Lie groups, we know that there are $3$-torsions of order $p$ only when $G=Spin(6)$, and $SU(i)$ with $2\leq i\leq 4$ in $\pi_{i+5}(G)$ for $1\leq i \leq 3$, and no other higher order $3$-torsions or $p$-torsions with $p\geq 5$. Hence by (\ref{expformu1eq2}) and (\ref{expformu1eq1}) 
\[{\rm exp}([\Sigma^i M, G])\leq p^{\epsilon(G, p)}\cdot {\rm exp}([\Sigma^i M_4, G])\leq p^{\epsilon(G, p)}\cdot {\rm max}\{ {\rm exp}(G\{c\}), {\rm exp}(G)\},\]
which justifies the lemma for the low dimensional homotopy groups.

The remaining part of the lemma follows immediately from Theorem \ref{gaugedecom6c}, \ref{gaugedecomPoincare} and Lemma \ref{mappingmooredes}.
\end{proof}

In order to examine the homotopy exponents of gauge groups via Lemma \ref{expformu1}, we need to study the related aspects of Lie groups. To set notation, let $X$ be a finite $p$-local $H$-space. Then by a classical result of Hopf, there is a rational homotopy equivalence 
\[
X\simeq_{(0)} S^{2n_1+1}\times S^{2n_2+1}\times \cdots \times S^{2n_l+1},
\]
where the index set $\mathfrak{t}(X)=\{n_1, n_2,\ldots, n_l\}$ ($n_1\leq n_2\leq \ldots\leq n_l$) is called the \textit{type} of $X$. Denote also $l(X)= n_l$. The types of compact simply connected simple Lie groups are well known and summarised in Table \ref{tabletypelie}.
\begin{table}[H]
\centering
\caption{Type of simple Lie groups}
\begin{tabular}{lp{3.7cm}|lp{3.7cm}lp{3.7cm}}
\hline
$G$      & Type  &   $G_2$        &   $1, 5$      \\ \hline      
$SU(n)$            & $1, 2, \ldots, n-1$      &   $F_4$        &   $1, 5, 7, 11$       \\ \hline
$Sp(n)$            & $1, 3, \ldots, 2n-1$          &   $E_6$        &   $1, 4, 5, 7, 8, 11$         \\ \hline
$Spin(2n)$            & $1, 3, \ldots, 2n-3, n-1$           &   $E_7$        &   $1, 5, 7, 9, 11, 13, 17$      \\ \hline
$Spin(2n+1)$            & $1, 3, \ldots, 2n-1$   &    $E_8$        &   $1, 7, 11, 13, 17, 19, 23, 29$             \\ \hline
\end{tabular}
\label{tabletypelie}
\end{table}

Due to a classical result of Serre \cite{Serre53}, it is well known that these Lie groups can be decomposed into product of spheres at large prime, which was generalized to finite $p$-local $H$-spaces by Kumpel \cite{Kumpel72}.

\begin{theorem}\label{regulardecomH}
Let $X$ be a finite ${\rm mod}~p$ $H$-space of type $\{n_1=1, n_2,\ldots, n_l\}$. If $H^\ast(X;\mathbb{Z})$ is $p$-torsion free and $p\geq l(X)+1$, then there exists a map 
\[
f: S^{2n_1+1}\times S^{2n_2+1}\times \cdots \times S^{2n_l+1}\rightarrow X,
\]
which is a ${\rm mod}~p$ homotopy equivalence.
\end{theorem}

The primes $p$ for which there exists a map $f$ as in Theorem \ref{regulardecomH} are called the \textit{regular primes} of $X$. With this terminology, Theorem \ref{regulardecomH} claims that $X$ is $p$-regular if $p\geq l(X)+1$.

On the other hand, thanks to a construction of Gray \cite{Gray69} on a family of infinitely many elements of order $p^n$ in $\pi_\ast(S^{2n+1})$, Cohen, Neisendorfer and Moore in their famous work \cite{Cohen79, Neisen3} showed that 
\begin{equation}\label{sphereexp}
{\rm exp}(\Omega^i_0S^{2n+1})=p^n,
\end{equation}
for any $i\geq0$.
\begin{lemma}\label{regularexplie}
Let $G$ be a Lie group in Table \ref{tabletypelie} with $p$-torsion free cohomology. Then if $p\geq l(G)+1$, we have 
\[
{\rm exp}(\Omega^i_0G)=p^{l(G)}, \ \ {\rm exp}(\Omega^i_0G\{c\})=p^{\nu_p(c)}
\]
for any $i\geq0$.
\hfill $\Box$
\end{lemma}

Theriault \cite{Theriault04} proved a useful lemma to study the homotopy exponent problem. For our purpose, we need two generalizations of his lemma, one of which immediately follows from the original one of Theriault.
\begin{lemma}[Lemma $2.2$ and $2.3$ of \cite{Theriault04}]\label{generalstephenlemma}
Suppose there is a homotopy fibration
\[
F\rightarrow E\rightarrow S^{2n_1+1}\times S^{2n_2+1}\times\cdots \times S^{2n_r+1},
\]
where $E$ is simply connected or an $H$-space and for each $1\leq i\leq r$
\[|{\rm coker}(\pi_{2n_i+1}(E)\rightarrow \pi_{2n_i+1}(S^{2n_i+1}))|\leq p^{t_i}.\]
Then 
\[
{\rm exp}(E)\leq p^t\cdot {\rm max}\{{\rm exp}(F), p^{n_1}, p^{n_2},\ldots, p^{n_r}\},
\]
where $t={\rm max}\{t_1, t_2, \ldots, t_r\}$.
\hfill $\Box$
\end{lemma}

For the homotopy exponents of $\mathcal{G}_k(P^4(c))$, we need to take the order of the connecting map $\partial_k$ into account, and by Proposition \ref{mooregaugecriterion} we actually label the gauge groups of $P^4(c)$ and $M$ by $k\in \mathbb{Z}/\nu_p(d)$ with $d=({\rm ord}(\partial_1), c)$.
\begin{proposition}\label{expformu2}
Let $G$ be a Lie group in Table \ref{tables4guageorder} in Section \ref{gaugemooresec}. Let $M$ as in Lemma \ref{expformu1}. Then if $p\geq l(G)+1$, we have for any $k\in\mathbb{Z}/\nu_p(d)$
\[
{\rm exp}(\mathcal{G}_k(M))\leq p^{\nu_p({\rm ord}(\partial_1))}\cdot {\rm max}\{p^{l(G)},  p^{\nu_p(c)}\},
\]
if $G\neq SU(i)$ with $i=2$, or $3$. For the the later two cases,
\[
{\rm exp}(\mathcal{G}_k(M))\leq p^{\nu_p({\rm ord}(\partial_1))+1}\cdot {\rm max}\{p^{l(G)},  p^{\nu_p(c)}\}.
\]

\end{proposition}
\begin{proof}
Since $G$ is $p$-regular, we can apply Lemma \ref{generalstephenlemma} to the homotopy fibre sequence 
\[\Omega^3_0G\{c\}\rightarrow \mathcal{G}_k(P^4(c))\rightarrow G\stackrel{\partial_k}{\rightarrow}\Omega_k^2G\{c\}.\]
Notice that 
\[|{\rm coker}(\pi_\ast(\mathcal{G}_k(P^4(c)))\rightarrow \pi_\ast(G))|\leq p^{\nu_p({\rm ord}(\partial_k))}.\]
Hence by Lemma \ref{regularexplie}, we see that 
\[
{\rm exp}(\mathcal{G}_k(P^4(c))\leq p^{\nu_p({\rm ord}(\partial_k))}\cdot {\rm max}\{p^{l(G)},  p^{\nu_p(c)}\}.
\]
But it is easy to check that $\nu_p(d)<p$ by the order of $\partial_1$ in Table \ref{tables4guageorder}, the lemma then follows from Lemma \ref{expformu1}.
\end{proof}

In order to get estimation of the homotopy exponents when $G$ is not $p$-regular, we need more general form of Lemma \ref{generalstephenlemma}. Indeed, Theriault \cite{Theriault07} proved a deep theorem of homotopy decompositions of the low rank Lie groups of Table \ref{tablelie2} with further applications of the homotopy exponents. 

\begin{table}[H]
\centering
\caption{Some low rank Lie groups}
\begin{tabular}{l|p{3.7cm}lp{3.7cm}}
\hline 
$SU(n)$            & $n-1\leq (p-1)(p-2)$        \\ \hline
$Sp(n)$            & $2n\leq (p-1)(p-2)$          \\ \hline
$Spin(2n+1)$            & $2n\leq (p-1)(p-2)$         \\ \hline
$Spin(2n)$            & $2(n-1)\leq (p-1)(p-2)$         \\ \hline
$G_2, F_4, E_6$            & $p\geq5$       \\ \hline
$E_7, E_8$            & $p\geq7$       \\ \hline
\end{tabular}
\label{tablelie2}
\end{table}

\begin{theorem}[Theriault \cite{Theriault07}]\label{Theriaultlieexp}
Let $G$ be a Lie group in Table \ref{tablelie2}. Then there is a homotopy commutative diagram
\[
\xymatrix{
\Omega G\ar[r]^{p^{r(G)}}  \ar[d]^{}   & \Omega G\ar@{=}[d]\\
\prod\limits_{k\in \mathfrak{t}(G)} \Omega S^{2k+1} \ar[r]^{\ \ \ \lambda}  & \Omega G,
}
\]
in which all maps are loop maps. In particular,
\[{\rm exp}(G) \leq p^{r(G)+l(G)}.\]
\end{theorem}

\begin{lemma}\label{newstephenlemma}
Let $G$ as in Table \ref{tablelie2}.
Suppose there is a homotopy fibration
\[
F\rightarrow E\rightarrow G,
\]
where for each $k\in \mathfrak{t}(G)$
\[|{\rm coker}(\pi_{2k+1}(E)\rightarrow \pi_{2k+1}(G))|\leq p^{t}.\]
Then 
\[
{\rm exp}(\Omega^2 E)\leq p^{t+r(G)}\cdot {\rm max}\{{\rm exp}(F), p^{r(G)+l(G)}\}.
\]
\end{lemma}
\begin{proof}
The proof follows the strategy of that of Lemma $2.2$ in \cite{Theriault04} but with a little more effects. First by the condition we can construct a homotopy commutative diagram 
\[
\xymatrix{
\prod\limits_{k\in \mathfrak{t}(G)} S^{2k+1} \ar[r]^{\ \  \ \  \ p^t} \ar[d]^{f} & G\ar@{=}[d] \\
E \ar[r]   & G, 
}
\]
which together with the diagram in Theorem \ref{Theriaultlieexp} can be fitted into a large homotopy commutative diagram 
\[
\xymatrix{
\Omega G\{p^{t+r(G)}\} \ar[r]  \ar[ddd]   &  \Omega G\ar[r]^{p^{t+r(G)}} \ar[d] &  \Omega G\ar@{=}[d] \\
&     \prod\limits_{k\in \mathfrak{t}(G)} \Omega S^{2k+1} \ar[d]^{\lambda} \ar[r]^{\ \ \ \ \ \lambda\cdot p^{t}}  &  \Omega G\ar@{=}[d]\\
 &     \prod\limits_{k\in \mathfrak{t}(G)} \Omega S^{2k+1} \ar[d]^{\Omega f} \ar[r]^{\ \ \ \ \  p^{t}}  &  \Omega G\ar@{=}[d]     \\
\Omega F \ar[r] &\Omega E \ar[r]  & \Omega G.                        
}
\]
The outer homotopy pullback diagram determines a homotopy fibration
\[
\Omega G\{p^{t+r(G)}\} \rightarrow \Omega G\times \Omega F \rightarrow \Omega E,
\]
which implies another fibration 
\[
 \Omega^2 G\times \Omega^2 F \rightarrow \Omega^2 E \rightarrow\Omega G\{p^{t+r(G)}\}.
\]
The lemma then follows from Lemma \ref{expxclemma}.
\end{proof}

\begin{proposition}\label{expformu3}
Let $G$ be a Lie group listed in Table \ref{tablelie2} (let $n\geq 5$ when $G=Spin(n)$). Let $M$ as in Lemma \ref{expformu1}. 
Then we have for any $k\in\mathbb{Z}/\nu_p(d)$
\[
{\rm exp}(\mathcal{G}_k(M))\leq p^{r(G)+\nu_p({\rm ord}(\partial_1))}\cdot {\rm max}\{p^{r(G)+l(G)},  p^{\nu_p(c)}\}.
\]
\end{proposition}
\begin{proof}
As the proof of Proposition \ref{expformu2}, we can apply Lemma \ref{newstephenlemma} to the homotopy fibre sequence 
\[\Omega^3_0G\{c\}\rightarrow \mathcal{G}_k(P^4(c))\rightarrow G\stackrel{\partial_k}{\rightarrow}\Omega_k^2G\{c\}.\]
Hence by Lemma \ref{Theriaultlieexp} and \ref{expxclemma}, we see that 
\[
{\rm exp}(\Omega^2 \mathcal{G}_k(P^4(c))\leq p^{r(G)+\nu_p({\rm ord}(\partial_1))}\cdot {\rm max}\{p^{r(G)+l(G)},  p^{\nu_p(c)}\}.
\]
The lemma then follows easily from Lemma \ref{expformu1} and its proof.
\end{proof}

Let us make several remarks. Firstly, we notice that Proposition \ref{expformu3} generalizes Proposition \ref{expformu2} which is the special case when $r(G)=0$. Secondly, the conditions of $(G, p)$ in Table \ref{tablelie2} are slightly stronger than those in Table \ref{tables4guageorder}. In particular, ${\rm ord}(\partial_1)$ can be estimated. Thirdly, the number $r(G)$ are computable. For instance, it is known $r(SU(n))=\nu_p((n-1)!)$. By a classical result of Legendre, 
\[
\nu_p((n-1)!)=\sum\limits_{k=1}^\infty \lfloor \frac{n-1}{p^k} \rfloor=\frac{n-1-s_p(n)}{p-1},
\]
where $s_p(n)=a_k+a_{k-1}+\ldots+a_1+a_0$ is the sum of all the digits in the expansion of $n$ in the base $p$. Hence 
\[
r(SU(n))=\nu_p((n-1)!)\leq  \lfloor \frac{n-2}{p-1} \rfloor,
\]
and then is much small than $l(SU(n))=n-1$. Since Harris \cite{Harris61} showed that 
\[
SU(2n)\simeq_{(p)} Sp(n)\times SU(2n)/Sp(n),~{\rm and}~ Spin(2n+1)\simeq_{(p)} Sp(n),
\]
we can get upper bounds of the homotopy exponents of $Sp(n)$ and $Spin(2n+1)$ from that of $SU(n)$.
Also, the homotopy exponents of exceptional Lie groups are known, as summarised in Theorem $1.10$ of \cite{DT08}. 
Lastly, one can work out concrete estimations by applying Proposition \ref{expformu2} and \ref{expformu3} with explicit values of $\tilde{\partial}_1$ in Table \ref{tables4guageorder} and of $r(G)$ in \cite{Theriault07}. The results are summarised in Corollary \ref{coroexpintro1} for the matrix groups and Table \ref{tableexp2intro} for the exceptional Lie groups.
\begin{table}[H]
\centering
\caption{Upper bounds of ${\rm exp}(\mathcal{G}_k(G)$ when $G$ is exceptional}
\begin{tabular}{lp{2.7cm}lp{3.7cm}lp{5.7cm}}
\hline
$G$          &       $p$               &      ${\rm exp}(\mathcal{G}_k(G))\leq ?$ \\ \hline
$G_2$      &     $p=5$           &    ${\rm max}(7, \nu_p(c)+1)$ \\
                 &     $p=7$           &    ${\rm max}(6, \nu_p(c)+1)$ \\
                 &     $p\geq11$           & ${\rm max}(5, \nu_p(c))$ \\ \hline  
$F_4$      &     $p=5$           &    ${\rm max}(15, \nu_p(c)+3)$ \\
                 &     $p=7$           &    ${\rm max}(13, \nu_p(c)+1)$ \\
                 &     $p=11$           & ${\rm max}(13, \nu_p(c)+1)$ \\
                  &     $p=13$           & ${\rm max}(12, \nu_p(c)+1)$ \\
                  &     $p\geq 17$       & ${\rm max}(11, \nu_p(c))$ \\ \hline
$E_6$      &     $p=5$           &    ${\rm max}(15, \nu_p(c)+3)$ \\
                 &     $p=7$           &    ${\rm max}(14, \nu_p(c)+2)$ \\
                 &     $p=11$           & ${\rm max}(13, \nu_p(c)+1)$ \\
                  &     $p=13$           & ${\rm max}(12, \nu_p(c)+1)$ \\
                  &     $p\geq 17$       & ${\rm max}(11, \nu_p(c))$ \\  \hline
$E_7$       &     $p=7$           &    ${\rm max}(22, \nu_p(c)+3)$ \\
                 &     $p=11$           & ${\rm max}(20, \nu_p(c)+2)$ \\
                  &     $p=13$           & ${\rm max}(19, \nu_p(c)+1)$ \\
                  &     $p=17$       & ${\rm max}(19, \nu_p(c)+1)$ \\                                        
                 &     $p=19$       & ${\rm max}(18, \nu_p(c)+1)$ \\  
                 &     $p\geq 23$       & ${\rm max}(17, \nu_p(c))$ \\ \hline
$E_8$       &     $p=7$           &    ${\rm max}(35, \nu_p(c)+4)$ \\
                 &     $p=11$           & ${\rm max}(33, \nu_p(c)+3)$ \\
                  &     $p=13$           & ${\rm max}(32, \nu_p(c)+2)$ \\
                  &     $p=17$       & ${\rm max}(31, \nu_p(c)+1)$ \\                                        
                 &     $p=19$       & ${\rm max}(32, \nu_p(c)+2)$ \\  
                 &     $p=23$       & ${\rm max}(31, \nu_p(c)+1)$ \\ 
                 &     $p=29$       & ${\rm max}(31, \nu_p(c)+1)$ \\          
                 &     $p=31$       & ${\rm max}(30, \nu_p(c)+1)$ \\ 
                 &     $p\geq37$       & ${\rm max}(29, \nu_p(c))$         
\end{tabular}
\label{tableexp2intro}
\end{table}

\section{Periodicity in the homotopy of gauge groups for stable bundles}\label{Bottsection}
\noindent In this section, we consider the homotopy of $\mathcal{G}_k(M)$ localized away from $c$, and then we can expect that we do not need divisibility conditions on $c$.
\begin{theorem}\label{Gauge-c} 
Let $M$ be a five dimensional oriented closed manifold with $\pi_{1}(M)\cong \mathbb{Z}/c$ ($c\neq 0$) and $H_{2}(M; \mathbb{Z})$ is torsion free of rank $m-1$. Let $G$ be a simply connected compact simple Lie group with $\pi_4(G)=0$. Then after localization away from $c$ we have the homotopy equivalences:
\begin{itemize}
\item if $M$ is a spin manifold,
\[\mathcal{G}_k(M)\simeq G\times \Omega^5 G\times \prod_{i=1}^{m-1} (\Omega^2 G  \times \Omega^3 G );\]
\item if $M$ is a non-spin manifold,
\[\mathcal{G}_k(M)\simeq G\times \Omega {\rm Map}_0^\ast(\mathbb{C}P^2, G)\times \prod_{i=1}^{m-1} \Omega^2 G  \times \prod_{i=1}^{m-2} \Omega^3 G,\]
\end{itemize}
where $k\in \mathbb{Z}/c$.
\end{theorem}
\begin{proof}
Recall in Section \ref{Sigma2Msection} we have the homotopy cofibre sequence (\ref{hdecomSM4})
\[P^4(c)\vee\bigvee_{j=1}^{m-1} S_j^3\stackrel{h}{\longrightarrow} P^3(c)\vee\bigvee_{i=1}^{m-1}S_i^3 \longrightarrow  \Sigma M_4
\]
such that $H_3(h)=0$. After localization from $c$ we see that $P^n(c)$ is contractible and 
\[\Sigma M_4 \simeq \bigvee_{i=1}^{m-1} (S^4\vee S^3).\]
Then by Lemma \ref{Sigmasplit}, we have the homotopy equivalence  
\[\Sigma M \simeq \bigvee_{i=1}^{m-2} (S^4\vee S^3)\vee \Sigma T,\]
where $T\simeq (S^3\vee S^2)\cup e^5$ as in Lemma \ref{gaugeL}. Then by similar argument there and in Lemma \ref{HtypeT} and Theorem \ref{gaugedecom6c} we can prove this theorem
(we also use the fact that $G$ splits off the gauge group of any trivial principal bundle).
\end{proof}

We want use Theorem \ref{Gauge-c} to study the periodicity phenomena in the homotopy groups of gauge groups for $G=SU(n)$ and $Spin(n)$. Let us denote $\mathcal{G}_k(G)=\mathcal{G}_k(M)$ to emphasize the group $G$ in these two cases.
Recall the Bott periodicity showed in the following table:
\begin{table}[H]
\centering
\captionof{table}{$\pi_r(SU(n))$ ($2n\geq r+1 $) and $\pi_r(Spin(n))$ ($n\geq r+2 \geq 4$)}
\begin{tabular}{c | c | c }
  $r~{\rm mod}~2$ &  $0$ & $1$  \\ \hline
$\pi_r(SU(n))$ &$0$ &$\mathbb{Z}$ 
\end{tabular}
\ \ \ 
\begin{tabular}{c | c | c | c | c | c | c | c | c }
 $r~{\rm mod}~8$ &  $0$ & $1$ & $2$ & $3$ & $4$ & $5$ & $6$ & $7$ \\ \hline
$\pi_r(Spin(n))$ &$\mathbb{Z}/2$ &$\mathbb{Z}/2$ & $0$ &$\mathbb{Z}$ & $0$ & $0$ &$0$ &$\mathbb{Z}$
\end{tabular}
\label{Bott}
\end{table}
Then it is easy to show that either for spin manifold $M$ after localization away from $c$ or for non-spin manifold $M$ after localization away from $2c$
\begin{equation}\label{stablegaugesu}
\pi_r(\mathcal{G}_k(SU(n)))\cong \oplus_{m}\mathbb{Z}
\end{equation}
for any $n\geq \frac{r}{2}+3$. Similarly, for $G=Spin(n)$ we need to consider two cases. When $M$ is spin, we have Table \ref{stablegaugespin1} after localization away from $c$:
\begin{table}[H]
\centering
\captionof{table}{$\pi_r(\mathcal{G}_k(Spin(n)))$ ($n\geq r+7\geq 9$) when $M$ is spin}
\begin{tabular}{c|c|c|c|c}
$r~{\rm mod}~8$ &  $0$ & $1$ & $2$ & $3$ \\ \hline
$\pi_r(\mathcal{G}_k(Spin(n)))$ &$\oplus_{m-1} \mathbb{Z}\oplus\mathbb{Z}/2$ 
&$\oplus_{m-1} \mathbb{Z}\oplus\mathbb{Z}/2$
&$\mathbb{Z}$
&$\mathbb{Z}\oplus\mathbb{Z}/2$
\\ \hline
&$4$ &  $5$ & $6$ & $7$ \\ \hline
&$\oplus_{m-1} \mathbb{Z}\oplus\mathbb{Z}/2$
&$\oplus_{m-1} (\mathbb{Z}\oplus\mathbb{Z}/2)$
&$\mathbb{Z}\oplus\oplus_{2m-2}\mathbb{Z}/2$
&$\mathbb{Z}\oplus\oplus_{m-1}\mathbb{Z}/2$ \\
\end{tabular}
\label{stablegaugespin1}
\end{table}
On the other hand, when $M$ is a non-spin manifold, we can take localization of $M$ away from $2c$ and get Table \ref{stablegaugespin2}.
\begin{table}[H]
\centering
\captionof{table}{$\pi_r(\mathcal{G}_k(Spin(n)))$ ($n\geq r+7\geq 9$) localized away from $2c$}
\begin{tabular}{c|c|c|c|c}
$r~{\rm mod}~4$ &  $0$ & $1$ & $2$ & $3$ \\ \hline
$\pi_r(\mathcal{G}_k(Spin(n)))$ &$\oplus_{m-1} \mathbb{Z}$
&$\oplus_{m-1} \mathbb{Z}$
&$\mathbb{Z}$
&$\mathbb{Z}$
\end{tabular}
\label{stablegaugespin2}
\end{table}


$\, $ 

\Acknowledgements{The author is supported by Postdoctoral International Exchange Program for Incoming Postdoctoral Students under Chinese Postdoctoral Council and Chinese Postdoctoral Science Foundation.
He is also supported in part by Chinese Postdoctoral Science Foundation (Grant No. 2018M631605), and National Natural Science Foundation of China (Grant No. 11801544).
\newline
$\, $ 
The author would like to thank Haibao Duan for bringing the non-simply connected $5$-manifolds to his attention. He is indebted to Daisuke Kishimoto for suggestions about adding more applications of the decompositions on an early version of this paper. He also wants to thank Fred Cohen, and the anonymous referees most warmly for their careful reading of the manuscript and many valuable suggestions which have improved the paper.}




\begin{appendix}
\section{\label{Qdecomsec}}
We complete the paper with an appendix to show that our methods can reprove the theorem of F\'{e}lix and Oprea \cite{FO09} for rational gauge groups from a different point of view, and also in a slightly larger context.

Let $X$ be a connected $CW$-complex of finite type such that $X$ is rationally simply connected. Denote the Hilbert series of $X$ by 
\[P_X(t)=\sum\limits_{i=0}^{\infty}b_it^i,\]
where $b_i={\rm dim}H^i(X; \mathbb{Q})$ is the $i$-th Betti number. Since rationally $\Sigma X$ is a wedge of spheres (see Theorem $24.5$ of \cite{FHT01}), we see that 
\[\Sigma X\simeq_{\mathbb{Q}} \bigvee_{i=1}^{\infty}\bigvee_{b_i}S^{i+1}.\]

For any principal bundle $G\rightarrow P\rightarrow X$ classified by $\alpha\in [X,BG]$, as we use $\mathcal{G}_\alpha(X)$ to denote $\mathcal{G}(P)$, let us denote $\mathcal{G}^{\ast}_\alpha(X)$ to be the based gauge groups $\mathcal{G}^\ast(P)$.
\begin{theorem}[cf. Theorem $2.3$ of \cite{FO09}]\label{Qgauge}
Let $X$ be a connected $CW$-complex of finite type such that $X$ is rationally simply connected and $G$ be any connected topological group with homotopy type of a $CW$-complex of finite type. 
There are rational homotopy equivalences for any $\alpha\in [X, BG]$
\[
\mathcal{G}_\alpha(X)\simeq_{\mathbb{Q}}\prod_{i=0}^{\infty}\prod_{b_i}\Omega^i G, \ \ \ \ 
\mathcal{G}^{\ast}_\alpha(X)\simeq_{\mathbb{Q}}\prod_{i=1}^{\infty}\prod_{b_i}\Omega^i G,
\]
with the convention $\Omega^0G=G$ and $\Omega^iG=\ast$ if $b_i=0$.
\end{theorem}
\begin{proof}
First by Atiyah-Bott \cite{AB83}, 
\[\mathcal{G}_{\alpha}^{\ast}(X)\simeq {\rm Map}^{\ast}_{0}(\Sigma X, BG).\]
Hence it is easy to see rationally 
\[\mathcal{G}^{\ast}_\alpha(X)\simeq_{\mathbb{Q}}\prod_{i=1}^{\infty}\prod_{b_i}\Omega^i G.\]

For the remaining part $\mathcal{G}_\alpha(X)$, the basic idea of the proof is similar to that of Theorem \ref{gaugedecom6c} and Theorem \ref{gaugedecomPoincare}, i.e., to apply the rational version of Proposition \ref{Gsplit} (for $i=1$) inductively.
At each stage we choose a inclusion $Y_{(k)}=S^n\rightarrow X_{(k)}$ for some bottom cell of $X_{(k)}$, and $Y_{(k-1)}\rightarrow X_{(k-1)}\rightarrow X_{(k)}$ is a cofibre sequence with $X_{(1)}=X$.

(To apply the induction, we only need to notice that the condition $[Y, BG]=0$ is not necessary here (which is only for tracing the components in the proof of Proposition \ref{Gsplit}), since here $Y$ is always a suspension and the components of ${\rm Map}_0^{\ast}(Y, BG)$ are homotopy equivalent to each other.) 

In particular, for each $k$ we get 
\[\mathcal{G}_\alpha(X)\simeq_{\mathbb{Q}} \prod_{i=1}^{k}\prod_{b_i}\Omega^i G\times \mathcal{G}_\alpha(X_k),\]
where $X_k$ is a space constructed by the inductive procedure. Then we have two maps
\[i: \mathcal{G}^{\ast}_\alpha(X)\simeq_{\mathbb{Q}}\prod_{i=1}^{\infty}\prod_{b_i}\Omega^i G\rightarrow \mathcal{G}_\alpha(X) \ \ \  r: \mathcal{G}_\alpha(X)\rightarrow  \prod_{i=1}^{\infty}\prod_{b_i}\Omega^i G,\]
such that $r\circ i\simeq {\rm id}$. Then the natural extension of topological groups 
\[1\rightarrow \mathcal{G}^{\ast}_\alpha(X)\rightarrow \mathcal{G}_\alpha(X)\rightarrow G\rightarrow1\]
splits rationally (as spaces), and we have completed the proof of the theorem.
\end{proof}
\begin{remark}
\begin{itemize}
\item We should notice that in the homotopy decompositions of $\mathcal{G}_\alpha(X)$ and $\mathcal{G}^{\ast}_\alpha(X)$ in Theorem \ref{Qgauge} and its proof, $\Omega^i G$ will be contractible if $G$ is with finite dimensional rational homology and $i$ is large enough. 
\item Theorem \ref{Qgauge} generalizes Theorem $2.3$ of \cite{FO09}, which is proved for finite bases and Lie groups.
\end{itemize}
\end{remark}

Suppose $P$ is a principal $G$-bundle classifying by $\alpha\in [X=M, BG]$, where $M$ is a closed oriented smooth Riemannian $n$-manifold. Let $\mathcal{A}(P)$ be the space of connections on $P$. Then $\mathcal{A}(P)$ admits a $\mathcal{G}(P)$-action by a standard pullback construction. However the action is not free in general. On the contrast, the based gauge group $\mathcal{G}^\ast(P)$ acts freely on the spaces of connections $\mathcal{A}(P)$ (e.g., see Section $1$ of \cite{CM94}) and there is a principle bundle 
\[\mathcal{G}^\ast(P)\rightarrow \mathcal{A}(P)\rightarrow \mathcal{B}^\ast(P)=\mathcal{A}(P)/\mathcal{G}^\ast(P),\]
which is indeed universal since $\mathcal{A}(P)$ is affine and then contractible. In particular,
\[\mathcal{B}^\ast(P)\simeq B\mathcal{G}^\ast(P)\simeq B{\rm Map}_{\alpha}(M, BG).\]
By the arguments in Section $15. f$ \cite{FHT01}, $BG$ is an $H$-spaces when $G$ is a path connected topological group with finite dimensional rational homology (which holds for all Lie groups).
Now by the similar argument in Theorem \ref{Qgauge} it is easy to determine the rational homotopy type of $\mathcal{B}^\ast(P)$. We may also denote $\mathcal{B}^\ast_\alpha(M)=\mathcal{B}^\ast(P)$.
\begin{corollary}\label{Qconnection}
Let $M$ as described, and $G$ be a path connected topological group with finite dimensional rational homology. 
There are rational homotopy equivalences for any $\alpha\in [M, BG]$
\[
\mathcal{B}^\ast_\alpha(M)\simeq_{\mathbb{Q}}\prod_{i=1}^{\infty}\prod_{b_i}\Omega^{i-1} G.
\]
\end{corollary}
\begin{proof}
The corollary follows easily by applying Proposition $15.15$ of \cite{FHT01} and the fact that $H$-spaces are formal and then are products of Eilenberg-Mac Lane spaces rationally (also see the argument right after the corollary). 
\end{proof}

It is a standard result in rational homotopy theory (Section $16.c$ of \cite{FHT01}) that any $H$-space is rationally homotopy equivalent to a product of Eilenberg-Mac Lane spaces. Moreover for our topological group $G$ with cohomology algebra 
\[H^\ast(G; \mathbb{Q})\cong \otimes_{i}\Lambda (a_i) \otimes \otimes_{j} \mathbb{Q}[b_j],\]
where $\Lambda (a_i)$ are exterior algebras with generators $a_i$ of degree $2n_i+1$, and $\mathbb{Q}[b_j]$ are polynomial algebras with generators $b_j$ of degree $2m_j$,
we have 
\begin{equation}\label{HEM}
G\simeq_{\mathbb{Q}} \prod_{i} S^{2n_i+1}\times \prod_{j} K(\mathbb{Q}, 2m_j).
\end{equation}
Hence by Theorem \ref{Qgauge} it is easy to get explicit descriptions of the rational homotopy types of $\mathcal{G}_\alpha(X)$ and $\mathcal{G}^{\ast}_\alpha(X)$. 
\begin{corollary}\label{QgaugeEM}
Let $X$ and $G$ as before. Then we have rationally homotopy equivalences 
\[
\mathcal{G}_\alpha(X)\simeq_{\mathbb{Q}}\prod_{\begin{subarray}{l}
i,j,k\\
k\geq0 \end{subarray}}
\big( \prod_{b_{2k+1}}(S^{2m_j-2k-1}\times
K(\mathbb{Q}, 2n_i-2k)) \times \prod_{b_{2k}} (S^{2n_i-2k+1} \times K(\mathbb{Q}, 2m_j-2k))\big),
\]
\[
\mathcal{G}^{\ast}_\alpha(X)\simeq_{\mathbb{Q}}\prod_{\begin{subarray}{l}
i,j,k\\
k\neq 0 \end{subarray}}
\big( \prod_{b_{2k+1}}(S^{2m_j-2k-1}\times
K(\mathbb{Q}, 2n_i-2k)) \times \prod_{b_{2k}} (S^{2n_i-2k+1} \times K(\mathbb{Q}, 2m_j-2k))\big).
\]
\end{corollary}

\begin{remark}
\begin{itemize}
\item With Corollary \ref{QgaugeEM} it is easy to check the rational groups of $\mathcal{G}_\alpha(X)$ and $\mathcal{G}^{\ast}_\alpha(X)$ satisfy the formulas (Theorem $3.1$ of \cite{FO09} by F\'{e}lix and Oprea)
\begin{eqnarray*}
\pi_{q}(\mathcal{G}_\alpha(X))\otimes \mathbb{Q} &\cong& \sum\limits_{r\geq 0}H^r(X; \mathbb{Q})\otimes \pi_{r+q}(G), \\  
\pi_{q}(\mathcal{G}^\ast_\alpha(X))\otimes \mathbb{Q} &\cong& \sum\limits_{r\geq 0}\tilde{H}^r(X; \mathbb{Q})\otimes \pi_{r+q}(G).
\end{eqnarray*}
Notice that our formulas hold for complexes of finite type and then can be applied to the universal $G$-bundle immediately (cf. Section $4$ of \cite{FO09} where F\'{e}lix and Oprea proved this using other method since their formulas hold only for finite complexes). Then as in \cite{FO09}, by a classic description of gauge groups, i.e.
\[\mathcal{G}_\alpha(X)\cong {\rm Map}_{G}(EG, G),\] 
we also get formulas for the homotopy groups of the homotopy fixed set $G^{hG}={\rm Map}_{G}(EG, G)$ but for any connected topological group of finite type instead of only Lie groups:
\[\pi_q(G^{hG})\otimes \mathbb{Q}\cong \sum\limits_{r\geq 0}H^r(BG; \mathbb{Q})\otimes \pi_{r+q}(G).\]
\item We can also get the cohomology rings of the gauge groups easily from Corollary \ref{QgaugeEM} and of their classifying spaces form Corollary \ref{Qconnection}. For instance, let $G$ be a connected topological group with finite dimensional rational homology (e.g., Lie groups; i.e., $j=0$ in (\ref{HEM}) for this case). then we have 
\[H^\ast(\mathcal{G}_\alpha(X); \mathbb{Q})\cong \otimes_{i,k}\big(\otimes_{b_{2k}}\Lambda (\mu_{i,k})\otimes \otimes_{b_{2k+1}}\mathbb{Q}[\nu_{i,k}]\big), \]  
where $\mu_{i,k}$ is of degree $2n_i-2k+1$ and $\nu_{i,k}$ is of degree $2n_i-2k$. Similarly, the cohomology ring of the orbit space of the space of connections is 
\[H^\ast(\mathcal{B}^\ast_\alpha(M); \mathbb{Q})\cong \otimes_{i,k}\big(
\otimes_{b_{2k+1}}\Lambda (x_{i,k})\otimes
\otimes_{b_{2k}}\mathbb{Q}[y_{i,k}] \big), \]  
where $x_{i,k}$ is of degree $2n_i-2k+1$ and $y_{i,k}$ is of degree $2n_i-2k+2$. 
\item As in \cite{FO09}, it is now easy to get the formulas for the rank of the homotopy groups, and in particular cover Terzi\'{c} Formula (Proposition $1$ and $2$ of \cite{Terzic05}) for simply connected closed $4$-manifolds. 
\end{itemize}
\end{remark}

\end{appendix}

\end{document}